\documentclass[a4paper]{article}

\usepackage{amsmath,amsthm,amssymb}
\usepackage{graphicx}
\usepackage{epstopdf}
\usepackage{listings}

\usepackage{tikz}
\tikzset{>=latex}

\def\centerarc[#1](#2)(#3:#4:#5);%
{
	\draw[#1]([shift=(#3:#5)]#2) arc (#3:#4:#5);
}

\DeclareMathOperator{\arccot}{arccot}
\DeclareMathOperator{\Res}{Res}
\DeclareMathOperator{\sgn}{sgn}

\newtheorem{lemma}{Lemma}[section]
\newtheorem{theorem}{Theorem}[section]
\newtheorem{corollary}{Corollary}[section]

\newcommand{\Ma}{\mathbf M_\alpha}

\newcommand{\R}{\mathbb{R}}

\newcommand{\Dag}{D_\gamma^\alpha}
\newcommand{\DD}{\mathcal{D}}
\newcommand{\Da}{\mathcal{D}^{\alpha}}
\newcommand{\Db}{\mathcal{D}^{\alpha-1}}

\newcommand{\barDa}{\overline{\mathcal{D}^{\alpha}}}
\newcommand{\barDb}{\overline{\mathcal{D}^{\alpha-1}}}
\newcommand{\FF}{\mathcal{F}}

\DeclareMathOperator{\erf}{erf}

\theoremstyle{remark}

\begin{document}

\title{Numerical Approximation of Riesz-Feller Operators on $\mathbb R$}

\author{Carlota M. Cuesta\footnotemark[2] \and Francisco de la Hoz\footnotemark[2] \and Ivan Girona\footnotemark[2]}

\date{
	\footnotesize
	$^1$Department of Mathematics, Faculty of Science and Technology, University of the Basque Country UPV/EHU, Barrio Sarriena S/N, 48940 Leioa, Spain
}

\maketitle

\begin{abstract}
In this paper, we develop an accurate pseudospectral method to approximate numerically the Riesz-Feller operator $D_\gamma^\alpha$ on $\R$, where $\alpha\in(0,2)$, and $|\gamma|\le\min\{\alpha, 2 - \alpha\}$. This operator can be written as a linear combination of the Weyl-Marchaud derivatives $\Da$ and $\barDa$, when $\alpha\in(0,1)$, and of $\partial_x\Db$ and $\partial_x\barDb$, when $\alpha\in(1,2)$.

Given the so-called Higgins functions $\lambda_k(x) = ((ix-1)/(ix+1))^k$, where $k\in\mathbb Z$, we compute explicitly, using complex variable techniques, $\Da[\lambda_k](x)$, $\barDa[\lambda_k](x)$, $\partial_x\Db[\lambda_k](x)$, $\partial_x\barDb[\lambda_k](x)$ and $D_\gamma^\alpha[\lambda_k](x)$, in terms of the Gaussian hypergeometric function ${}_2F_1$, and relate these results to previous ones for the fractional Laplacian. This enables us to approximate $\Da[u](x)$, $\barDa[u](x)$, $\partial_x\Db[u](x)$, $\partial_x\barDb[u](x)$ and $D_\gamma^\alpha[u](x)$, for bounded continuous functions $u(x)$. Finally, we simulate a nonlinear Riesz-Feller fractional diffusion equation, characterized by having front propagating solutions whose speed grows exponentially in time.

\end{abstract}

\medskip

\noindent\textit{Keywords:}

\noindent Weyl-Marchaud derivatives, Riesz-Feller operators, Pseudospectral methods, Riesz-Feller fractional diffusion equations

\section{Introduction}

In this paper, we occupy ourselves with the so-called Riesz-Feller operators. A general Riesz-Feller operator $\Dag$ of order $\alpha\in(0,2)$ and skewness $\gamma\in\mathbb R$, with $|\gamma| \leq \min\{\alpha, 2-\alpha\}$, can be defined by means of a Fourier multiplier operator (see, e.g., \cite{mainardiluchkopagnini2001}):
\begin{equation}\label{RF:operator}
\FF(\Dag[u])(\xi) = \psi^\alpha_\gamma(\xi) \FF(u)(\xi),
\end{equation}
where the Fourier symbol is given by
\begin{equation}\label{RF:symbol}
\psi^\alpha_\gamma(\xi) = - |\xi|^\alpha e^{-i \sgn(\xi) \gamma\frac\pi2}.
\end{equation}
Recall that $\FF$, which denotes the Fourier transform, is defined as
\begin{equation*}
\FF(u)(\xi) = \hat{u}(\xi)= \frac{1}{\sqrt{2\pi}} \int_{-\infty}^\infty{u(x) e^{-i\xi x} dx} \Longleftrightarrow u(x) = \frac{1}{\sqrt{2\pi}} \int_{-\infty}^\infty{\hat u(\xi) e^{i\xi x} d\xi}.
\end{equation*}
We note that the definition we use here for Riesz-Feller operators differs from the one in \cite{mainardiluchkopagnini2001}, because such definition uses the complex conjugate of $\FF(u)(\xi)$ in the definition of the Fourier transform (up to a scaling factor).

The following integral representations, given in \cite[Proposition~2.3]{AchleitnerKuehn2015} (see also \cite{Cifani,mainardiluchkopagnini2001,Sato}), are an alternative way of dealing with these pseudo-differential operators. More precisely, defining (see, e.g., \cite{mainardiluchkopagnini2001})
\begin{equation*}
c^1_{\gamma,\alpha} = \frac{\Gamma(1+\alpha)}{\pi}\sin\left( (\alpha-\gamma) \frac\pi2\right) \quad  \text{and} \quad c^2_{\gamma,\alpha} = \frac{\Gamma(1+\alpha)}{\pi}\sin\left( (\alpha+\gamma) \frac\pi2\right),
\end{equation*}
which satisfy $c^1_{\gamma,\alpha} + c^2_{\gamma,\alpha} >0$, we distinguish between the following cases:

\begin{itemize}
	
\item For $\alpha\in(0,1)$ and $|\gamma| \leq \alpha$,
\begin{equation}
\label{RF:01:integral}
\Dag[u](x) = c^1_{\gamma,\alpha} \int_{0}^{\infty} \frac{u(x-z) -u(x)}{z^{1+\alpha}} dz+ c^2_{\gamma,\alpha} \int_{0}^{\infty} \frac{u(x+z) -u(x)}{z^{1+\alpha}} dz.
\end{equation}

\item For $\alpha\in(1,2)$ and $|\gamma| \leq 2 - \alpha$,
\begin{align}
\label{RF:integral}
\Dag[u](x) & = c^1_{\gamma,\alpha} \int_{0}^{\infty} \frac{u(x-z) -u(x) + u'(x) z}{z^{1+\alpha}} dz 
 + c^2_{\gamma,\alpha} \int_{0}^{\infty} \frac{u(x+z) -u(x) - u'(x) z}{z^{1+\alpha}} dz.
\end{align}

\item For $\alpha\in(0,2)$ and $\gamma=0$, we have $c_{0,\alpha}^1=c_{0,\alpha}^2$, so we obtain minus the fractional Laplacian (see \cite{kwasnicki}):
\begin{equation}
	\label{e:fraclap}
	(-\Delta)^{\alpha/2}u(x) = c_\alpha\int_{-\infty}^\infty\frac{u(x)-u(x+y)}{|y|^{1+\alpha}}dy,
\end{equation}
where the integral (and, in general, all the integrals over $\mathbb R$ in this paper) must be understood in the principal value sense, and
\begin{equation}
\label{e:ca}
c_\alpha = \alpha\frac{2^{\alpha-1}\Gamma(\frac12 + \frac\alpha2)}{\sqrt{\pi}\Gamma(1-\frac\alpha2)}.
\end{equation}
Indeed,
\begin{equation}
\label{e:equivDzeroafraclap}
D_0^\alpha\equiv -(-\Delta)^{\alpha/2},
\end{equation}
since 
$$
c_{0,\alpha}^1=\frac{\Gamma(1+\alpha)}{ \pi}\sin\left(\alpha\frac\pi2\right)=\alpha\frac{2^{\alpha-1}\Gamma(\frac12 + \frac\alpha2)}{\sqrt{\pi}\Gamma(1-\frac\alpha2)}=c_\alpha,
$$
which can be checked by using Euler's reflection formula:
\begin{equation}
	\label{Euler:R}
	\Gamma(\omega)\Gamma(1-\omega)=  \frac{\pi}{\sin(\omega \pi)},
\end{equation}
and Legendre's duplication formula:
\begin{equation}
	\label{e:legendre}
	\Gamma(\omega)\Gamma\left(\omega+\frac12\right)=2^{1-2\omega}\sqrt{\pi}\Gamma(2\omega),
\end{equation}
taking $\omega=\alpha/2$ in both cases.	

\item For $\alpha=1$ and $|\gamma|\leq 1$ (see, e.g., \cite{mainardiluchkopagnini2001}),
\begin{equation*}
 D^1_\gamma[u](x) = -\cos \left(\gamma\frac\pi2\right)\frac{d}{dx}\mathcal H(u(x)) +\sin \left(\gamma\frac\pi2\right) u'(x),
\end{equation*}
where $\mathcal H(u(x))$ is the Hilbert transform of $u(x)$:
$$
\mathcal H(u(x)) = \frac1\pi\int_{-\infty}^\infty\frac{u(z)}{x-z}dz.
$$
Recall that, $(d/dx)\mathcal H(u(x)) = \mathcal H(u'(x))$. Moreover, the derivative of the Hilbert transform is equivalent to $(-\Delta)^{1/2}$, which is the fractional Laplacian \eqref{e:fraclap} for $\alpha = 1$, and is also known as the half Laplacian \cite{CuestaDelaHozGirona2023}. Therefore, we have
\begin{equation}
\label{e:D1g}
D^1_\gamma[u](x) = -\cos \left(\gamma\frac\pi2\right)(-\Delta)^{1/2}u(x) +\sin \left(\gamma\frac\pi2\right) u'(x).
\end{equation}

\end{itemize}

On the other hand, the limits $\alpha\to 2^-$ and $\alpha\to 0^+$ imply necessarily $|\gamma|\to 0$, and we formally have the second derivative in the former case, and the identity in the latter case. Observe that, as with the fractional Laplacian, the latter is a singular limit:
\begin{align*}
(-\Delta)^{0}1 & = 1, \text{ but } (-\Delta)^{\alpha/2}1=0,\text{ if }\alpha>0,
	\cr
\DD^0[1] & = 1 \text{ and } \overline{\DD^0}[1]=-1,  \text{ but } \DD^\alpha[1]=0 \text{ and } \overline{\DD^\alpha}[1]=0,\text{ if }\alpha > 0.
\end{align*}
A related operator, known in the literature of fractional derivatives as a right Weyl-Marchaud fractional derivative of order $\alpha$ (see \cite{Marchaud,Weyl}), is given by
\begin{equation}\label{e:Dax}
\Da[u](x) = \frac{1}{\Gamma(-\alpha)}
\int_{-\infty}^{0}\frac{u(x+z)-u(x)}{|z|^{1+\alpha}} dz,\quad \alpha\in(0,1),
\end{equation}
where $u\in\mathcal C^0_b(\mathbb R)$;  note that $\mathcal C_b^r(\mathbb R)$ denotes the set of bounded $r$-times continuous functions defined on $\mathbb R$.

Our motivation is to study certain models where such operators appear, and that give rise to moving front solutions, such as \cite{SK1985,Fowler,NV,achleitnerhittmeirschmeiser2011,achleitnercuestahittmeir,AchleitnerKuehn2015,AU,diezcuesta2020,cuestadiaz2023}. For example, the operator $\partial_x\Da$,  where $\partial_x\Da[\cdot](x) \equiv (d/dx)(\Da[\cdot](x))$, appears in the study of certain boundary-layer equations derived from shallow-water models \cite{NV} and other related models, such as those described in \cite{SK1985,Fowler}. The mathematical analysis of such equations can be found in, e.g., \cite{achleitnerhittmeirschmeiser2011,achleitnercuestahittmeir,AchleitnerKuehn2015,AU,diezcuesta2020,cuestadiaz2023}, and a numerical method for $\partial_x\Da$, for $\alpha\in(0,1)$, is derived in \cite{delahozcuesta2016}. Note that $\partial_x\Da$ is of Riesz-Feller type with order $1+\alpha$ and skewness $\gamma = 1-\alpha$, as we will explain below.

It is known, see \cite{CabreRoquejoffre2013}, that nonlocal Fisher equations generate exponentially fast accelerating front solutions. We remark that the class of operators considered in \cite{CabreRoquejoffre2013} contains Riesz-Feller operators in one dimension; therefore, this model will serve us as a good test for the numerical methods of evolution equations.

We observe that \eqref{RF:01:integral} and \eqref{RF:integral} can be written in terms of Weyl-Marchaud derivatives and their derivatives, which is done in \cite{diezcuesta2020,cuestadiaz2023}. For the sake of completeness, we recall the key lemmas.
\begin{lemma}[Equivalent representations of $\Da$ and $\partial_x \Da$]\label{equiv:repre}
 Let $\alpha \in (0,1)$, $u \in\mathcal C_b^1(\R)$ and $x\in \R$. Then,
\begin{equation}\label{Int:Dalpha}
\Da[u](x) = \frac{1}{\Gamma(1-\alpha)} \int^{\infty}_{0} \frac{u'(x-z)}{z^\alpha} dz.
\end{equation}
Moreover, if $u \in \mathcal C^2_b(\R)$,
\begin{align}
\label{e:xDax}
  \partial_x \Da[u](x) &=
  \frac{1}{\Gamma(-1-\alpha)} \int^{\infty}_{0} \frac{u(x-z)-u(x)+ u'(x)z}{z^{2+\alpha}} dz,
  \\
  \partial_x \Da[u](x) &=\frac{1}{\Gamma(-\alpha)}\int^{\infty}_{0} \frac{u'(x-z)-u'(x)}{z^{1+\alpha}} dz,
  \cr
\nonumber	\partial_x \Da[u](x) & = \frac{1}{\Gamma(1-\alpha)}\int_0^\infty\frac{u''(x-z)}{z^\alpha}dz.
\end{align}
\end{lemma}
\begin{proof}
 The first two identities have been obtained in \cite{cuestadiaz2023}. In order to obtain the third one, we first change the variable $z$ to $-z$, and then apply the fundamental theorem of calculus and Fubini:
\begin{align*}
&  \int_{-\infty}^{0}
  \frac{ u(x+z)-u(x)- u'(x) z}{ |z|^{2+\alpha}} dz
  = 
  \int_{-\infty}^{0}
  \frac{ \int_{0}^{z} u'(x+y) dy - u'(x)z }{ (-z)^{2+\alpha} } dz \\
  & \qquad = \int_{-\infty}^{0} \int_{0}^{z}
  \frac{ u'(x+y)- u'(x) }{ (-z)^{2+\alpha}} dy dz
  = \int_{-\infty}^{0} \int_{-\infty}^{y}
  \frac{u'(x)-u'(x+y)}{(-z)^{2+\alpha}} dz dy\\
  & \qquad = \frac{1}{1+\alpha} \int_{-\infty}^{0}
  \frac{u'(x)-u'(x+y)}{(-y)^{1+\alpha}} dy = \frac{1}{1+\alpha} \int^{\infty}_{0}
  \frac{u'(x)-u'(x-z)}{z^{1+\alpha}} dz,
\end{align*}
and the third identity follows after applying that $(-1-\alpha)\Gamma(-1-\alpha)=\Gamma(-\alpha)$. Finally, the fourth identity is proved in a similar way:
\begin{align*}
	\partial_x \Da[u](x) & = \frac{1}{\Gamma(-\alpha)}\int^{\infty}_{0} \frac{u'(x-z)-u'(x)}{z^{1+\alpha}}dz = -\frac{1}{\Gamma(-\alpha)}\int^{\infty}_{0} \frac{\int_0^zu''(x-y)dy}{z^{1+\alpha}}dz
	\cr
	& = -\frac{1}{\Gamma(-\alpha)}\int_0^\infty\int_y^\infty\frac{u''(x-y)}{z^{1+\alpha}}dzdy = \frac{1}{\Gamma(1-\alpha)}\int_0^\infty\frac{u''(x-y)}{y^\alpha}dy,
\end{align*}
and the identity follows after changing the integration variable $y$ by $z$.
\end{proof}
Now, we consider the operator $\barDa$, defined by means of
\begin{equation}
\label{e:barDax}
\barDa[u](x)= - \frac{1}{\Gamma(-\alpha)} \int_{0}^{\infty} \frac{u(x+z)-u(x)}{z^{1+\alpha}} dz, \quad \alpha\in(0,1),
\end{equation}
which is minus the left Weyl-Marchaud derivative of order $\alpha$ (see \cite{Marchaud,Weyl}). Then, we have the following representations:

\begin{lemma}[Equivalent representations of $\barDa$ and $\partial_x \barDa$]\label{equiv:repre:bar}
 Let $\alpha \in (0,1)$, $u \in\mathcal C_b^1(\R)$ and $x\in \R$. Then,
\begin{equation}\label{Int:Dalpha:bar}
\barDa[u](x) = \frac{1}{\Gamma(1-\alpha)} \int_{0}^{\infty} \frac{u'(x+z)}{z^\alpha} dz.
\end{equation}
Moreover, if $u\in \mathcal C^2_b(\R)$,
\begin{align}
\label{e:xbarDax}
\partial_x \barDa[u](x) &= \frac{1}{\Gamma(-1-\alpha)} \int_{0}^{\infty} \frac{u(x+z)-u(x)- u'(x)z}{z^{2+\alpha}} dz,
	\\
\partial_x \barDa[u](x) &	= -\frac{1}{\Gamma(-\alpha)}\int_0^\infty\frac{u'(x+z)-u'(x)}{z^{1+\alpha}}dz,
	\cr
\nonumber\partial_x \barDa[u](x) & = \frac{1}{\Gamma(1-\alpha)}\int_0^\infty\frac{u''(x+z)}{z^\alpha}dz.
\end{align}
\end{lemma}

\begin{proof} Analogous to the proof of Lemma \ref{equiv:repre}.
\end{proof}

Note that, from \eqref{RF:01:integral}, \eqref{e:Dax} and \eqref{e:barDax}, it follows that
\begin{equation}
	\label{RF:01:integral:representation}
	\Dag[u](x) = \Gamma(-\alpha)\left(c_{\gamma,\alpha}^1\Da[u](x)
	- c_{\gamma,\alpha}^2 \barDa[u](x)\right), \quad \alpha\in(0,1);
\end{equation}
and, from \eqref{RF:integral}, \eqref{e:xDax} and \eqref{e:xbarDax},
\begin{align}
\label{RF:integral:representation}
\Dag[u](x) & = \Gamma(-\alpha)\left(c_{\gamma,\alpha}^1\partial_x\DD^{\alpha-1}[u](x)
+ c_{\gamma,\alpha}^2 \partial_x\overline{\DD^{\alpha-1}}[u](x)\right), \quad \alpha \in(1, 2).
\end{align}
It is now quite straightforward to check that \eqref{RF:01:integral:representation} and \eqref{RF:integral:representation} are indeed consistent with the definition by the Fourier symbol \eqref{RF:operator}-\eqref{RF:symbol}. In order to do it, we first need the Fourier symbol of $\Da$ (see, e.g., \cite[Chapter~7]{SKM}):
\begin{align*}
\FF\left( \Da[u](x) \right)(\xi) & = (i \xi)^\alpha  \FF(u)(\xi) = |\xi|^\alpha e^{i\sgn(\xi)\alpha\frac\pi2} \FF(u)(\xi) 
	\cr
& = -|\xi|^\alpha e^{-i\sgn(\xi)(2-\alpha)\frac\pi2} \FF(u)(\xi), \quad\alpha\in(0,1),
\end{align*}
and that of $\barDa$:
\begin{equation*}
\FF\left( \barDa[u](x) \right)(\xi) = -(-i\xi)^\alpha  \FF(u)(\xi) = -|\xi|^\alpha e^{-i\sgn(\xi)\alpha\frac\pi2} \FF(u)(\xi), \quad \alpha\in(0,1),
\end{equation*}
i.e., their Fourier symbols satisfy $(i\xi)^\alpha=- \overline{(-(-i\xi)^\alpha)}$, where the bar on the right-hand side denotes complex conjugation. We can observe that $\Da$ is not of Riesz-Feller type, because its symbol has order $\alpha\in(0,1)$ and skewness $\gamma=2-\alpha$, so it does not satisfy $|\gamma|\le\alpha$; however, $\barDa$ does belong to this class, since it has order $\alpha\in(0,1)$ and skewness $\gamma=\alpha$.

On the other hand,
\begin{align*}
\FF(\partial_x \Da[u])(\xi) & = (i\xi)\FF(\Da[u])(\xi) = (i\xi)^{1+\alpha} \FF(u)(\xi) = |\xi|^{1+\alpha}e^{i\sgn(\xi)(1+\alpha)\frac\pi2} \FF(u)(\xi) 
	\cr
& = -|\xi|^{1+\alpha}e^{-i\sgn(\xi)(1-\alpha)\frac\pi2} \FF(u)(\xi), \quad \alpha\in(0,1),
\end{align*}
which means that $\partial_x \Da$ is an operator of Riesz-Feller type with order $1+\alpha\in(1,2)$ and skewness $\gamma = 1-\alpha$. Moreover, $\partial_x\barDa$ is also of Riesz-Feller type:
\begin{align*}
\FF(\partial_x \barDa[u])(\xi) & = (i\xi)\FF(\barDa[u])(\xi) = (-i\xi)^{1+\alpha} \FF(u)(\xi) = |\xi|^{1+\alpha}e^{-i\sgn(\xi)(1+\alpha)\frac\pi2} \FF(u)(\xi) 
	\cr
	& = -|\xi|^{1+\alpha}e^{-i\sgn(\xi)(-1+\alpha)\frac\pi2} \FF(u)(\xi), \quad \alpha\in(0,1),
\end{align*}
so it has order $1 + \alpha\in(1,2)$ and skewness $\gamma = -1+\alpha$. Note that the Fourier symbols of $\FF(\partial_x \Da[u])(\xi)$ and $\FF(\partial_x \barDa[u])(\xi)$ satisfy $(i\xi)^{1+\alpha} = \overline{(-i\xi)^{1+\alpha}}$,

We observe also that, when $\alpha\not\in\mathbb Z$,
\begin{align}\label{check:symbol}
& \Gamma(-\alpha) \left[c_{\gamma,\alpha}^1 (i\xi)^\alpha + c_{\gamma,\alpha}^2 (-i\xi)^\alpha\right] 
	\cr
& \qquad = \Gamma(-\alpha) \left[\frac{\Gamma(1+\alpha)}{\pi}\sin\left( (\alpha-\gamma) \frac\pi2\right) (i\xi)^\alpha + \frac{\Gamma(1+\alpha)}{\pi}\sin\left( (\alpha+\gamma) \frac\pi2\right) (-i\xi)^\alpha\right]
	\cr
& \qquad = -\frac{|\xi|^\alpha}{\sin(\pi\alpha)} \Big[\sin\left( \alpha \frac\pi2\right)\cos\left( \gamma \frac\pi2\right)\left(e^{i\sgn(\xi)\alpha\frac\pi2} + e^{-i\sgn(\xi)\alpha\frac\pi2}\right) 
	\cr
& \qquad \qquad - \cos\left( \alpha \frac\pi2\right)\sin\left( \gamma \frac\pi2\right) \left(e^{i\sgn(\xi)\alpha\frac\pi2} - e^{-i\sgn(\xi)\alpha\frac\pi2}\right)\Big] = -|\xi|^\alpha e^{-i\sgn(\xi)\gamma\frac\pi2},
\end{align}
where we have used that $(i\xi)^\alpha = |\xi|^\alpha e^{i\sgn(\xi)\alpha\frac\pi2}$, and \eqref{Euler:R}, for $\omega = -\alpha$. Then, applying the Fourier transform $\FF$ in both sides of \eqref{RF:01:integral:representation} and bearing in mind all the previous arguments,
\begin{align*}
\FF(\Dag[u])(\xi) & = \Gamma(-\alpha)\left(c_{\gamma,\alpha}^1\FF(\Da[u])(\xi) - c_{\gamma,\alpha}^2 \FF(\barDa[u](\xi))\right)
	\cr
& = \Gamma(-\alpha)\left(c_{\gamma,\alpha}^1(i \xi)^\alpha  \FF(u)(\xi) + c_{\gamma,\alpha}^2 (-i\xi)^\alpha  \FF(u)(\xi)\right)
	\cr
& = -|\xi|^\alpha e^{-i\sgn(\xi)\gamma\frac\pi2}\FF(u)(\xi), \quad\alpha\in(0,1).
\end{align*}
Likewise, applying the Fourier transform $\FF$ in both sides of \eqref{RF:integral:representation},
\begin{align*}
\FF(\Dag[u])(\xi) & = \Gamma(-\alpha)\left(c_{\gamma,\alpha}^1\FF(\partial_x\DD^{\alpha-1}[u])(\xi) + c_{\gamma,\alpha}^2 \FF(\partial_x\overline{\DD^{\alpha-1}}[u])(\xi)\right)
	\cr
& = \Gamma(-\alpha)\left(c_{\gamma,\alpha}^1(i\xi)^{1+(\alpha-1)} \FF(u)(\xi) + c_{\gamma,\alpha}^2 (-i\xi)^{1+(\alpha-1)} \FF(u)(\xi)\right)
	\cr
& = -|\xi|^\alpha e^{-i\sgn(\xi)\gamma\frac\pi2}\FF(u)(\xi), \quad\alpha\in(1,2).
\end{align*}
Therefore, \eqref{RF:01:integral:representation} and \eqref{RF:integral:representation} are consistent with \eqref{RF:operator}-\eqref{RF:symbol}, when $\alpha\in(0,1)\cup(1,2)$.

From now on, we will adopt the representation formulas \eqref{RF:01:integral:representation} and \eqref{RF:integral:representation}, where $\Da[u](x)$ and $\barDa[u](x)$, for $\alpha\in(0,1)$, are given respectively by \eqref{Int:Dalpha} and \eqref{Int:Dalpha:bar}, and $\partial_x\DD^{\alpha-1}[u](x)$ and $\partial_x\overline{\DD^{\alpha-1}}[u](x)$, for $\alpha\in(1,2)$, are given respectively by \eqref{e:xDax} and \eqref{e:xbarDax}.

The structure of this paper is as follows. In Section \ref{s:DaDabarlambdak}, using complex analysis techniques, we express $\Da[\varphi_k](x)$, $\partial_x\Da[\varphi_k](x)$, $\barDa[\varphi_k](x)$ and $\partial_x\barDa[\varphi_k](x)$, for $k\in2\mathbb Z$, in terms of the Gaussian hypergeometric function ${}_2F_1$  (see, e.g., \cite{slater1966}), where
\begin{equation}
\label{e:phik}
\varphi_k(x) \equiv \left(\frac{ix - 1}{ix + 1}\right)^{k/2} \equiv \frac{(x + i)^k}{(1 + x^2)^{k/2}}.
\end{equation}
These functions arise naturally when considering the change of variable $x = \cot(s)$, with $s\in(0,\pi)$, because $\varphi_k(\cot(s)) = e^{iks}$, for all $k\in\mathbb Z$. Of particular relevance are the functions with $k$ even, that we denote $\lambda_k(x)$:
\begin{equation}
\label{e:lk}
\lambda_k(x) \equiv \varphi_{2k}(x) \equiv \left(\frac{ix - 1}{ix + 1}\right)^{k} \equiv \frac{(x + i)^{2k}}{(1 + x^2)^k}, \quad k\in\mathbb Z.
\end{equation}
These functions are known in the literature as the complex Higgins functions (see \cite{boyd1990,higgins,narayan}), and form a complete orthogonal system in $\mathcal L^2(\mathbb R)$ with weight $w(x) = 1/(\pi(1+x^2))$. Moreover, in \cite{cayamacuestadelahoz2020}, $(-\Delta)^{\alpha/2}\lambda_k(x)$ was expressed in terms of ${}_2F_1$. In this regard, by using complex variable techniques, we prove in this paper that, for all $k\in\mathbb Z$,
\begin{equation}
\label{e:dalphalkvsflalk}
\begin{split}
	\Da[\lambda_{k}](x) & \equiv e^{-i\sgn(k)\alpha\frac\pi2}(-\Delta)^{\alpha/2}\lambda_{k}(x), \quad\alpha\in(0,1),
	\cr
	\barDa[\lambda_{k}](x) & \equiv -e^{i\sgn(k)\alpha\frac\pi2}(-\Delta)^{\alpha/2}\lambda_{k}(x), \quad\alpha\in(0,1),
	\cr
	\partial_x\DD^{\alpha-1}[\lambda_{k}](x) & \equiv e^{-i\sgn(k)\alpha\frac\pi2}(-\Delta)^{\alpha/2}\lambda_{k}(x), \quad\alpha\in(1,2),
	\cr
	\partial_x\overline{\DD^{\alpha-1}}[\lambda_{k}](x) & \equiv e^{i\sgn(k)\alpha\frac\pi2}(-\Delta)^{\alpha/2}\lambda_{k}(x), \quad\alpha\in(1,2).
\end{split}
\end{equation}
These nontrivial identities have deep implications, because they enable us to compute, e.g., the Weyl-Marchaud derivatives of the so-called complex Christov functions $\mu_k(x) = (ix - 1)^k/(ix + 1)^{k+1}$ (see \cite{boyd1990,christov,narayan,wiener}), which we offer for the sake of completeness. More importantly, the identities in \eqref{e:dalphalkvsflalk} allow  to extend the results on $(-\Delta)^{\alpha/2}\lambda_{k}(x)$ to the Weyl-Marchaud derivatives of $\lambda_k(x)$. In particular, in \cite{cayamacuestadelahoz2021}, after applying the change of variable $x = \cot(s)$, with $s\in(0,\pi)$, $(-\Delta)^{\alpha/2}\varphi_{k}(x)$ was expressed in terms of $s$, for all $k\in\mathbb Z$, in a way that was very convenient from a numerical point of view. Therefore, bearing in mind \eqref{e:dalphalkvsflalk}, we give analogous formulas for the Weyl-Marchaud derivatives of $\lambda_k(x)$ in terms of $s$, which will enable us to compute numerically $\Da[\lambda_k](x)$, $\partial_x\Da[\lambda_k](x)$, $\barDa[\lambda_k](x)$ and $\partial_x\barDa[\lambda_k](x)$ efficiently and accurately. Nevertheless, the identities \eqref{e:dalphalkvsflalk} are not valid for $\varphi_k(x)$ with $k$ odd, as is shown in Section \ref{s:Dphik}.

Section \ref{s:RFln} is devoted to the Riesz-Feller operator $D_\gamma^\alpha$ acting on $\lambda_k(x)$, which is expressed in terms of ${}_2F_1$, as well as in terms of $s$, in an analogous way to that in \cite{cayamacuestadelahoz2021}. Moreover, explicit expressions for $D_\gamma^1[\varphi_k](x)$, with $k$ odd, are also obtained, and Riesz-Feller operators acting on $\mu_k(x)$ are briefly studied.

In Section \ref{s:numerical}, we develop a pseudospectral method analogous to that in \cite{cayamacuestadelahoz2021}, to compute the Weyl-Marchaud derivatives and the Riesz-Feller operators applied to bounded continuous functions $u(x)$, such that the limits of $u(x)$ coincide as $x\to\pm\infty$. Furthermore, we show how to apply the method also to functions $u(x)$, such that $\lim_{x\to-\infty}u(x) \not= \lim_{x\to\infty}u(x)$, and offer full Matlab \cite{matlab} codes that enable to reproduce the results. Moreover, we also test the codes to unbounded functions, obtaining coherent results. Finally, we simulate the evolution of the following Fisher-type nonlinear Riesz-Feller fractional diffusion equation:
\begin{equation}
	\label{e:fisher0}
u_t = D_\gamma^\alpha u + f(u), \quad x\in\mathbb R,\ t \ge 0,
\end{equation}
for $f(u) = u(1-u)$, and show numerically the appearance of traveling fronts whose speed grows exponentially, being able to capture correctly this behavior.

\section{Weyl-Marchaud derivatives of $\lambda_k(x)$}

\label{s:DaDabarlambdak}

First of all, we review some familiar concepts. Let $z\in\mathbb C$, then, the Pochhammer symbol is defined, when $n\in\mathbb N$ (where $\mathbb N$ denotes the set of natural numbers without $0$), as $(z)_n \equiv z(z+1)\ldots(z+n-1)$, and when $n = 0$, as $(z)_0\equiv1$. Therefore, if $z\in\mathbb Z^-\cup\{0\}$ (where $\mathbb Z^-$ denotes the set of negative integers), and $n > |z|$, then $(z)_n = 0$. Furthermore, when $z\not\in\mathbb Z^-\cup\{0\}$, it can be expressed in terms of the Gamma function:
\begin{equation*}
(z)_n=\frac{\Gamma(z+n)}{\Gamma(z)}, \quad n\in\mathbb N.
\end{equation*}
The Pochhammer symbol intervenes in many identities, e.g.,
\begin{equation*}
(z)_n = (-1)^n (-z-n+1)_n= (-1)^n \frac{\Gamma(-z+1)}{\Gamma(-z-n+1)}= n!(-1)^n \binom{-z}{n}, \quad n\in\mathbb N.
\end{equation*}
Moreover, it appears in the definition of hypergeometric functions, in particular, of the Gaussian hypergeometric function ${}_2F_1$:
\begin{equation}
\label{e:2F1}
{}_{2}F_1(a,b;c;z)=\sum_{n=0}^{\infty}\frac{(a)_n(b)_n}{(c)_n}\frac{z^n}{n!}.
\end{equation}
In this paper, we will mainly encounter the two following particular cases of ${}_2F_1$:
\begin{equation*}
{}_2F_1(-m, 1+\alpha; 1; z) = \sum_{n=0}^{m}\frac{(-m)_n(1+\alpha)_n}{(1)_n}\frac{z^n}{n!} = \sum_{n=0}^m\binom{m}{n}\binom{-1-\alpha}{n}z^n
\end{equation*}
and
\begin{align}\label{hyper:2}
{}_2F_1(-m, 1+\alpha; 2; z)  = \sum_{n=0}^{m}\frac{(-m)_n(1+\alpha)_n}{(2)_n}\frac{z^n}{n!} = -\frac{1}{\alpha}\sum_{n=0}^m\binom{m}{n}\binom{-\alpha}{n+1}z^n,
\end{align}
which can be written as finite sums, because $m\in\mathbb N\cup\{0\}$, so they pose no problems related to absolute convergence. Note that, when $m = 0$, ${}_2F_1(0, 1+\alpha; 1; z) = {}_2F_1(0, 1+\alpha; 2; z) = 1$.

The function ${}_2F_1$ appears in a multitude of places. For instance, the fractional Laplacian of the complex Higgins functions  $(-\Delta)^{\alpha/2}\lambda_k(x)$ can be expressed \cite[Th. 2.1]{cayamacuestadelahoz2020} in terms of ${}_2F_1$, for $\alpha\in(0,2)$, $k\in\mathbb{Z}$ and $x\in\mathbb{R}$:
\begin{equation}
\label{e:fraclap2F1}
(-\Delta)^{\alpha/2}\lambda_{k}(x) = -\frac{2|k|\Gamma(1+\alpha)}{(i\sgn(k)x+1)^{1+\alpha}}{}_2F_1\left(1 - |k|, 1 + \alpha; 2; \frac{2}{i\sgn(k)x+1}\right);
\end{equation}
note that, when $\alpha = 1$, \eqref{e:fraclap2F1} can be simplified to
\begin{equation}
\label{e:fraclapa1}
(-\Delta)^{1/2}\lambda_{k}(x) = i\sgn(k)\lambda_{k}'(x) = \frac{2|k|}{1+x^2}\lambda_k(x)  = \frac{2|k|}{1+x^2}\left(\frac{ix - 1}{ix + 1}\right)^{k} = \frac{2|k|(x+i)^{2k}}{(1+x^2)^{1+k}}.
\end{equation}
On the other hand, ${}_2F_1$ also enables us to express $\Da[\lambda_k] (x)$ and $\barDa[\lambda_k](x)$ in a compact way that closely resembles \eqref{e:fraclap2F1}, as proved in the following theorem.
\begin{theorem}\label{t:th1}
Let $\lambda_k(x)$ be defined as in \eqref{e:lk}, $\alpha\in(0,1)$, $k\in\mathbb{Z}$ and $x\in\mathbb{R}$. Then,
\begin{align}
\label{Dalpha:ln:F}
\Da[\lambda_{k}](x) & \equiv e^{-i\sgn(k)\alpha\frac\pi2}(-\Delta)^{\alpha/2}\lambda_{k}(x)
\cr
& = -\frac{2|k|\Gamma(1+\alpha)e^{-i\sgn(k)\alpha\frac\pi2}}{(i\sgn(k)x+1)^{1+\alpha}}{}_2F_1\left(1 - |k|, 1 + \alpha; 2; \frac{2}{i\sgn(k)x+1}\right),
\\
\label{Dalpha:ln:bar:F}
\barDa[\lambda_{k}](x) & \equiv -e^{i\sgn(k)\alpha\frac\pi2}(-\Delta)^{\alpha/2}\lambda_{k}(x)
\cr
& = \frac{2|k|\Gamma(1+\alpha)e^{i\sgn(k)\alpha\frac\pi2}}{(i\sgn(k)x+1)^{1+\alpha}}{}_2F_1\left(1 - |k|, 1 + \alpha; 2; \frac{2}{i\sgn(k)x+1}\right).
\end{align}
Moreover, when $\alpha\in(1,2)$, $k\in\mathbb{Z}$ and $x\in\mathbb{R}$,
\begin{align}
	\label{Dalpha:ln:Fa12}
	\partial_x\DD^{\alpha-1}[\lambda_{k}](x) & \equiv e^{-i\sgn(k)\alpha\frac\pi2}(-\Delta)^{\alpha/2}\lambda_{k}(x)
	\cr
	& = -\frac{2|k|\Gamma(1+\alpha)e^{-i\sgn(k)\alpha\frac\pi2}}{(i\sgn(k)x+1)^{1+\alpha}}{}_2F_1\left(1 - |k|, 1 + \alpha; 2; \frac{2}{i\sgn(k)x+1}\right),
	\\
	\label{Dalpha:ln:bar:Fa12}
	\partial_x\overline{\DD^{\alpha-1}}[\lambda_{k}](x) & \equiv e^{i\sgn(k)\alpha\frac\pi2}(-\Delta)^{\alpha/2}\lambda_{k}(x)
	\cr
	& = -\frac{2|k|\Gamma(1+\alpha)e^{i\sgn(k)\alpha\frac\pi2}}{(i\sgn(k)x+1)^{1+\alpha}}{}_2F_1\left(1 - |k|, 1 + \alpha; 2; \frac{2}{i\sgn(k)x+1}\right).
\end{align}

\end{theorem}

\begin{proof}
We first assume $\alpha\in(0,1)$ and $k\in\mathbb N$. The derivative of \eqref{e:lk} is
$$
\lambda_k'(x) = -\frac{2ik}{(x - i)^2}\left(\frac{x + i}{x - i}\right)^{k-1}.
$$
Substituting $u'$ by this expression on the right-hand sides of \eqref{Int:Dalpha} and \eqref{Int:Dalpha:bar}, we get
\begin{align}
\label{Dalpha:ln}\Da[\lambda_k] (x)=-\frac{2ik}{\Gamma(1-\alpha)}\int^{\infty}_{0} \frac{(z-x - i)^{k-1}}{z^\alpha(z-x + i)^{k+1}} dz = -\frac{2ik}{\Gamma(1-\alpha)}\int^{\infty}_{0} g_k^0(z;x) dz,
	\\
\label{Dalpha:ln:bar}\barDa[\lambda_k](x) = -\frac{2ik}{\Gamma(1-\alpha)} \int_{0}^{\infty} \frac{(z+x + i)^{k-1}}{ z^\alpha(z+x - i)^{k+1} }dz = -\frac{2ik}{\Gamma(1-\alpha)} \int_{0}^{\infty}g_k^1(z;x)dz,
\end{align}
where we denote
$$
g_k^0(z;x)=\frac{(z-x - i)^{k-1}}{z^\alpha(z-x + i)^{k+1}}, \qquad g_k^1(z;x)=\frac{(z+x + i)^{k-1}}{z^\alpha(z+x - i)^{k+1}};
$$
observe that, in this notation, we regard $x$ as a parameter, rather than as an independent variable.

The integrals in \eqref{Dalpha:ln} and \eqref{Dalpha:ln:bar} can be computed by standard contour integration. More precisely, we integrate $g_k^j(z;x)$, for $j \in\{0, 1\}$, along certain integration contours on $\mathbb{C}$, and use Cauchy's integral theorem. Since $z^{\alpha} = e^{\alpha\ln(z)} = e^{\alpha(\ln(|z|) + i\arg(z))}$, $z^{\alpha}$ has a branch cut. In what follows, we consider the principal branch of the logarithm, which corresponds to $-\pi < \arg(z) \le \pi$; in particular, $(-1)^\alpha = e^{i\pi\alpha}$, unless the branch cut is crossed. The branch choice determines also how we choose the contours.

In Figure \ref{f:contour}, we have depicted one such contour $C$, which consists of four parts, i.e., $C = C_1 \cup C_R\cup C_2\cup C_r$, avoids the branch cut, but encloses the poles $z = i - x$ and $z = x - i$, for a given $x\in\mathbb R$. Then, by the residue theorem, the integral along it is equal to the sum of the residues. The pieces of the contour that run parallel to the branch cut will give the approximation of the integral from $0$ to $\infty$; the other pieces will give integrals that tend to zero, when $C$ tends to one contour that encloses $\mathbb{C}$, except for the brach cut. More precisely, for every $x\in\mathbb{R}$, we take $r$ and $R$ (i.e., the radii of $C_r$ and $C_R$, respectively), such that $r = 1/R$ and $0<r\ll (1+|x|^2)^{1/2}\ll R$. Moreover, we also take angles $\theta_1,\theta_2\in(\pi/2,\pi)$, such that $\delta = r\sin(\theta_1) = R\sin(\theta_2)$. In fact, if we fix, e.g, $\theta_1$, then the whole contour can be unequivocally determined by the value of $R$. The parameterization of the parts of $C$ is as follows:
\begin{align*}
  C_1 & = \{ -y-i\delta : y\in (-r\cos(\theta_1),-R\cos(\theta_2))\},\\
  C_R & = \{Re^{\theta i}: \theta \in (-\theta_2,\theta_2)\},\\
  C_2 & = \{ y+i\delta : y\in (R\cos(\theta_2),r\cos(\theta_1))\},\\
  C_r & = \{ re^{-\theta i}: \theta \in (-\theta_1,\theta_1)\}.
\end{align*}
Hence, for both $j=0$ and $j = 1$,
\begin{equation}
\label{e:intgg}
\int_C g_k^j(z;x)dz  = \int_{C_1}g_k^j(z;x)dz + \int_{C_R}g_k^j(z;x)dz + \int_{C_2}g_k^j(z;x)dz + \int_{C_r}g_k^j(z;x)dz.
\end{equation}
On the other hand, by Cauchy's residue theorem, we have:
\begin{equation}
\label{e:cauchy}
\begin{aligned}
\int_C g_k^0(z;x)dz & = 2\pi i \Res(g_k^0(z;x),x-i),
  \cr
\int_C g_k^1(z;x)dz & = 2\pi i \Res(g_k^1(z;x),i-x).
\end{aligned}
\end{equation}

\begin{figure}
	\centering
	\begin{tikzpicture}
		%small circle
		\centerarc[thick,->,blue](0,0)(158.2:-158.2:0.5cm);
		%C1
		\filldraw[->, blue, thick] (-3, 0.191342) -- (-0.441940, 0.191342);
		%C2
		\filldraw[->, blue, thick] (-0.451940, -0.191342) -- (-3, -0.191342);
		%big circle
		\centerarc[thick,<-,blue](0,0)(176.53:-176.53:3cm);
		%axes
		\draw (0,-3.5) -- (0, 3.5);
		\draw (-3.5, 0) -- (3.5, 0);
		%nodes
		\draw(2.8,1.9) node{\scriptsize $C_R$};
		\draw(-2, 0.5) node{\scriptsize $C_2$};
		\draw(0.65,0.45) node{\scriptsize $C_r$};
		\draw(-2, -0.5) node{\scriptsize $C_1$};
		\draw(1.4142, -1.4142)  node{ \scriptsize $x-i$};
		\draw(-1.4142, 1.4142) node{\scriptsize $i-x$};
		\draw(2.88, 2.98) node{\scriptsize $x>0$};
		%points
		\foreach \Point in {(-1.06066, 1.06066) ,(1.06066, -1.06066) }{
			\node at \Point {\textbullet};}
	\end{tikzpicture}
	\caption{An integration contour example, with the two poles $z = i - x$ and $z = x + i$, corresponding to a given $x>0$.}
	\label{f:contour}
\end{figure}
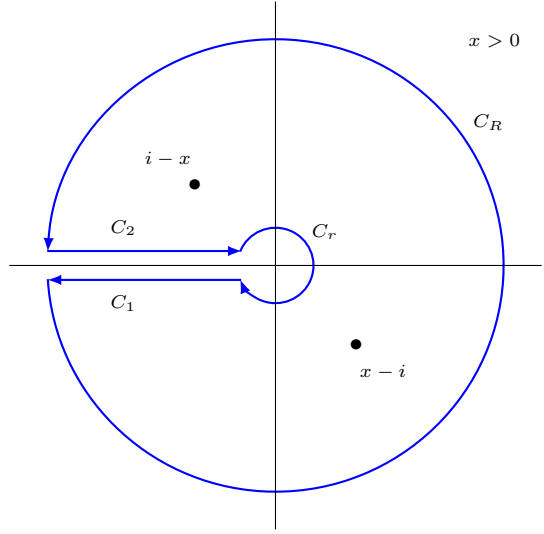

\noindent In order to compute the residues of $g_k^j(z;x)$ at $z = i - x$ and $z = x - i$, respectively, we use the general Leibniz rule:
$$
\frac{d^k}{dz^k}[p(z)q(z)] = \sum_{n=0}^{k}\binom{k}{n}p^{(n)}(z)q^{(k-n)}(z).
$$
Thus, we obtain
\begin{align}
\label{e:resi-x}
&\Res(g_k^1(z;x),i-x)  =
\left.\frac{1}{k!}\frac{d^k}{dz^k}z^{-\alpha}(z+x+i)^{k-1}\right|_{z = i-x}
\cr
& \qquad= \left.\frac{1}{k!}\sum_{n = 1}^k\binom{k}{n}\left[n!\binom{-\alpha}{n}z^{-n-\alpha}\right]\left[\frac{(k-1)!}{(k-1-(k-n))!}(z+x+i)^{k-1-(k-n)}\right]\right|_{z = i-x}
\cr
&\qquad = \sum_{n = 1}^k\binom{k-1}{n-1}\binom{-\alpha}{n}(i-x)^{-n-\alpha}(2i)^{n-1}
\cr
&\qquad = -\frac{\alpha e^{-i\pi(1+\alpha)/2}}{(ix+1)^{1+\alpha}}{}_2F_1\left(1 - k, 1 + \alpha; 2; \frac{2}{ix+1}\right),
\end{align}
where we have rewritten the sum, so $n$ goes from $0$ to $k - 1$ in the sum, and applied \eqref{hyper:2}. Likewise, we have
\begin{align}
\label{e:resx-i}
\Res(g_k^0(z;x),x-i) & = \sum_{n = 1}^k\binom{k-1}{n-1}\binom{-\alpha}{n}(x-i)^{-n-\alpha}(-2i)^{n-1}
	\cr
& = -\frac{\alpha e^{i\pi(1+\alpha)/2}}{(ix+1)^{1+\alpha}}{}_2F_1\left(1 - k, 1 + \alpha; 2; \frac{2}{ix+1}\right).
\end{align}
In order to compute \eqref{e:intgg}, we observe that the second and fourth integrals tend to zero, as $R\to\infty$ (this can be easily seen by changing to polar coordinates and recalling that $r=1/R$). In what regards the first and third integrals, we have that $C_1$ tends to $(-\infty+i\,0^-,0)$, parameterized by $-y+i\,0^-$, with $y\in(0,\infty)$, and $C_2$ tends to $(-\infty+i\,0^+,0)$, parameterized by $y+i\,0^+$, with $y\in(-\infty,0)$; note that the notation $+i\,0^+$ and $+i\,0^-$ allows us to indicate correctly the argument of the complex numbers lying on the negative real axis, which is $-\pi$ on $C_1$, and $\pi$ on $C_2$. Hence,
\begin{align*}
\lim_{R\to \infty}\int_{C_1}g_k^j(z;x)dz & = -\int_{0}^\infty g_k^j(-y + i\,0^-;x)dy = -e^{i\pi\alpha}\int_{0}^\infty g_k^{1-j}(y;x) dy,
	\cr
\lim_{R\to \infty}\int_{C_2}g_k^j(z;x)dz & = \int_{-\infty}^0 g_k^j(y + i\,0^+;x)dy = \int_0^{\infty} g_k^j(-y + i\,0^+;x)dy 
 = e^{-i\pi\alpha}\int_{0}^\infty g_k^{1-j}(y;x) dy.
\end{align*}
Therefore, 
\begin{align*}
\int_C g_k^j(z;x)dz & = \lim_{R\to\infty}\int_{C_1}g_k^j(z;x)dz + \lim_{R\to\infty}\int_{C_2}g_k^j(z;x)dz
	\cr
& = (-e^{i\pi\alpha} + e^{-i\pi\alpha})\int_{0}^\infty g_k^{1-j}(y;x)dy = -2i\sin(\pi\alpha)\int_{0}^\infty g_k^{1-j}(y;x)dy,
\end{align*}
i.e.,
\begin{equation*}
\int_{0}^\infty g_k^j(y;x)dy = \frac{i}{2\sin(\pi\alpha)}\int_C g_k^{1-j}(z;x)dz.
\end{equation*}
Taking $j \in\{0, 1\}$, from \eqref{e:cauchy}, together with \eqref{e:resi-x} and \eqref{e:resx-i}, respectively, we get
\begin{align*}
\int_{0}^\infty g_k^0(y;x)dy & = \frac{i}{2\sin(\pi\alpha)}\int_C g_k^{1}(z;x)dz = -\frac{\pi}{\sin(\pi\alpha)}\Res(g_k^1(z;x),x-i)
\cr
& = \frac{\pi\alpha e^{-i\pi(1+\alpha)/2}}{\sin(\pi\alpha)(ix+1)^{1+\alpha}}{}_2F_1\left(1 - k, 1 + \alpha; 2; \frac{2}{ix+1}\right),
\cr
\int_{0}^\infty g_k^1(y;x)dy & = \frac{i}{2\sin(\pi\alpha)}\int_C g_k^{0}(z;x)dz = -\frac{\pi}{\sin(\pi\alpha)}\Res(g_k^0(z;x),i-x)
\cr
& = \frac{\pi\alpha e^{i\pi(1+\alpha)/2}}{\sin(\pi\alpha)(ix+1)^{1+\alpha}}{}_2F_1\left(1 - k, 1 + \alpha; 2; \frac{2}{ix+1}\right).
\end{align*}
Then, from \eqref{Dalpha:ln} and \eqref{Dalpha:ln:bar}, respectively,
\begin{align}
\label{e:dalkpos}
\Da[\lambda_k](x) & = -\frac{2k\Gamma(1+\alpha)e^{-i\alpha\frac\pi2}}{(ix+1)^{1+\alpha}}{}_2F_1\left(1 - k, 1 + \alpha; 2; \frac{2}{ix+1}\right),
	\\
\label{e:dabarlkpos}
\barDa[\lambda_k](x) & = \frac{2k\Gamma(1+\alpha)e^{i\alpha\frac\pi2}}{(ix+1)^{1+\alpha}}{}_2F_1\left(1 - k, 1 + \alpha; 2; \frac{2}{ix+1}\right),
\end{align}
where we have used Euler's reflection formula \eqref{Euler:R}. The case with $k\in\mathbb Z^-$ follows from the symmetry $\lambda_{k}(x) \equiv \overline{\lambda_{-k}(x)}$. Hence, from \eqref{Int:Dalpha} and \eqref{Int:Dalpha:bar}, respectively,
$$
\Da[\lambda_{k}](x) \equiv \Da\Big[\overline{\lambda_{-k}}\Big](x) \equiv \overline{\Da[\lambda_{-k}](x)}, \qquad \barDa[\lambda_{k}](x) \equiv \barDa\Big[\overline{\lambda_{-k}}\Big](x)  \equiv \overline{\barDa[\lambda_{-k}](x)},
$$
for all $k\in\mathbb Z^-$. Therefore, from \eqref{e:dalkpos} and \eqref{e:dabarlkpos}, respectively,
\begin{align*}
\Da[\lambda_{k}](x) & = \frac{2k\Gamma(1+\alpha)e^{i\alpha\frac\pi2}}{(-ix+1)^{1+\alpha}}{}_2F_1\left(1 + k, 1 + \alpha; 2; \frac{2}{-ix+1}\right), \quad k\in\mathbb Z^-,
\cr
\barDa[\lambda_{k}](x) & = -\frac{2k\Gamma(1+\alpha)e^{-i\alpha\frac\pi2}}{(-ix+1)^{1+\alpha}}{}_2F_1\left(1 + k, 1 + \alpha; 2; \frac{2}{-ix+1}\right), \quad k\in\mathbb Z^-.
\end{align*}
Putting the cases with $k\in\mathbb Z^+$ and with $k\in\mathbb Z^-$ together,
\begin{align*}
\Da[\lambda_{k}](x) & = -\frac{2|k|\Gamma(1+\alpha)e^{-i\sgn(k)\alpha\frac\pi2}}{(i\sgn(k)x+1)^{1+\alpha}}{}_2F_1\left(1 - |k|, 1 + \alpha; 2; \frac{2}{i\sgn(k)x+1}\right),
	\cr
\barDa[\lambda_{k}](x) & = \frac{2|k|\Gamma(1+\alpha)e^{i\sgn(k)\alpha\frac\pi2}}{(i\sgn(k)x+1)^{1+\alpha}}{}_2F_1\left(1 - |k|, 1 + \alpha; 2; \frac{2}{i\sgn(k)x+1}\right).
\end{align*}
Then, bearing in mind \eqref{e:fraclap2F1}, this concludes the proof of \eqref{Dalpha:ln:F} and \eqref{Dalpha:ln:bar:F}.

Let us consider now the case with $\alpha\in(1,2)$. Then,
\begin{align*}
& \partial_x\DD^{\alpha-1}[\lambda_{k}](x) = \frac{d}{dx}\left[-\frac{2|k|\Gamma(\alpha)e^{-i\sgn(k)\pi(\alpha-1)/2}}{(i\sgn(k)x+1)^{\alpha}}{}_2F_1\left(1 - |k|, \alpha; 2; \frac{2}{i\sgn(k)x+1}\right)\right]
	\cr
& \qquad = -2ik\Gamma(\alpha)e^{-i\sgn(k)\alpha\frac\pi2}\frac{d}{dx}\sum_{n=0}^{|k|-1}\frac{(1-|k|)_n(\alpha)_n}{(2)_n}\frac{2^n(i\sgn(k)x+1)^{-n-\alpha}}{n!}
	\cr
& \qquad = -2|k|\Gamma(\alpha)e^{-i\sgn(k)\alpha\frac\pi2}\sum_{n=0}^{|k|-1}\frac{(1-|k|)_n(n+\alpha)(\alpha)_n}{(2)_n}\frac{2^n(i\sgn(k)x+1)^{-1-n-\alpha}}{n!}
	\cr
& \qquad = -\frac{2|k|\Gamma(1+\alpha)e^{-i\sgn(k)\alpha\frac\pi2}}{(i\sgn(k)x+1)^{1+\alpha}}{}_2F_1\left(1 - |k|, 1 + \alpha; 2; \frac{2}{i\sgn(k)x+1}\right),
\end{align*}
where we have used that $(n+\alpha)(\alpha)_n \equiv \alpha(1+\alpha)_n$. Then, bearing in mind \eqref{e:fraclap2F1}, this concludes the proof of \eqref{Dalpha:ln:Fa12}. The proof of \eqref{Dalpha:ln:bar:Fa12} is identical: we just replace $-e^{-i\sgn(k)\pi(\alpha-1)/2} = -i\sgn(k)e^{-i\sgn(k)\pi\alpha/2}$ by $e^{i\sgn(k)\pi(\alpha-1)/2} = -i\sgn(k)e^{i\sgn(k)\pi\alpha/2}$.

\end{proof}

Theorem \ref{t:th1} allows to compute immediately the Weyl-Marchaud derivatives of the so-called complex Christov functions \cite{boyd1990,christov,narayan,wiener}, which we offer here for the sake of completeness. These functions are defined as
\begin{equation}
\label{e:muk}
\mu_k(x) \equiv \frac{(ix - 1)^k}{(ix + 1)^{k+1}} \equiv \dfrac{\lambda_{k}(x)-\lambda_{k+1}(x)}{2}, \quad k\in\mathbb Z,
\end{equation}
and form a complete orthogonal system in $\mathcal L^2(\mathbb R)$. Moreover, they are closely related Malmquist-Takenaka functions \cite{Malmquist1926,Takenaka1925}, which differ from the Christov functions in a scaling and a constant factor.

In order to compute the Weyl-Marchaud derivatives of $\mu_k(x)$, we recall the expression of their fractional Laplacian \cite[Prop. 3.1]{cayamacuestadelahoz2020}, for $\alpha\in(0,2)$, $k\in\mathbb{Z}$ and $x\in\mathbb{R}$:
\begin{equation}
\label{e:fraclapmuk}
	(-\Delta)^{\alpha/2}\mu_k(x) =
	\left\{
	\begin{aligned}
		& \frac{\Gamma(1+\alpha)}{(ix+1)^{1+\alpha}}{}_2F_1\left(-k,1+\alpha; 1;\dfrac{2}{ix+1}\right), & & k\in\mathbb N\cup\{0\},
		\\
		& {-\frac{\Gamma(1+\alpha)}{(-ix+1)^{1+\alpha}}}{}_2F_1\left(1+k,1+\alpha; 1;\dfrac{2}{-ix+1}\right), & & k\in\mathbb Z^-.
	\end{aligned}
	\right.
\end{equation}
Then, the following corollary follows from Theorem \ref{t:th1} and \eqref{e:muk}.
\begin{corollary}

\label{cor:Damuk}

Let $\mu_k(x)$ be defined as in \eqref{e:muk},  $k\in\mathbb{Z}$ and $x\in\mathbb{R}$. Then,
\begin{equation}
	\label{e:dalphamukvsflamuk}
	\begin{split}
		\Da[\mu_{k}](x) & \equiv e^{-i\sgn(k)\alpha\frac\pi2}(-\Delta)^{\alpha/2}\mu_{k}(x), \quad\alpha\in(0,1),
		\cr
		\barDa[\mu_{k}](x) & \equiv -e^{i\sgn(k)\alpha\frac\pi2}(-\Delta)^{\alpha/2}\mu_{k}(x), \quad\alpha\in(0,1),
		\cr
		\partial_x\DD^{\alpha-1}[\mu_{k}](x) & \equiv e^{-i\sgn(k)\alpha\frac\pi2}(-\Delta)^{\alpha/2}\mu_{k}(x), \quad\alpha\in(1,2),
		\cr
		\partial_x\overline{\DD^{\alpha-1}}[\mu_{k}](x) & \equiv e^{i\sgn(k)\alpha\frac\pi2}(-\Delta)^{\alpha/2}\mu_{k}(x), \quad\alpha\in(1,2),
	\end{split}
\end{equation}
where $(-\Delta)^{\alpha/2}\mu_k(x)$ is given by \eqref{e:fraclapmuk}.

\end{corollary}

\begin{proof}
Straightforward. For instance, when $\alpha\in(0,1)$, from \eqref{Dalpha:ln:F} and \eqref{e:muk},
\begin{align*}
\Da[\mu_{k}](x) & = \Da\left[\frac{\lambda_k - \lambda_{k+1}}{2}\right](x) = \frac{\Da[\lambda_{k}](x) - \Da[\lambda_{k+1}](x)}{2}
	\cr
& = \frac{e^{-i\sgn(k)\alpha\frac\pi2}(-\Delta)^{\alpha/2}\lambda_{k}(x) - e^{-i\sgn(k+1)\alpha\frac\pi2}(-\Delta)^{\alpha/2}\lambda_{k+1}(x)}{2}.
	\cr
& = e^{-i\sgn(k)\alpha\frac\pi2}(-\Delta)^{\alpha/2}\frac{\lambda_{k}(x) - \lambda_{k+1}(x)}{2} = e^{-i\sgn(k)\alpha\frac\pi2}(-\Delta)^{\alpha/2}\mu_k(x),
\end{align*}
where we have used that, when $k \not = 0$ and $k \not = -1$, $\sgn(k) \equiv \sgn(k+1)$, when $k = 0$, $(-\Delta)^{\alpha/2}\lambda_{k}(x) = (-\Delta)^{\alpha/2}1 = 0$, and, when $k = -1$, $(-\Delta)^{\alpha/2}\lambda_{k+1}(x) = (-\Delta)^{\alpha/2}1 = 0$.

The proof of the other identities in \eqref{e:dalphamukvsflamuk} is identical.

\end{proof}

Note that the identities \eqref{e:dalphalkvsflalk} and \eqref{e:dalphamukvsflamuk} allow obtaining immediately explicit formulas for the Weyl-Marchaud derivatives of the sine-like and cosine-like versions of $\lambda_k(x)$ and $\mu_k(x)$. We omit them here, in order not to increase excessively the length of this paper, but they are a straightforward adaptation of the results in \cite{cayamacuestadelahoz2020} and pose no difficulty.

On the other hand, in \cite[Th. 2.2]{cayamacuestadelahoz2020}, $(-\Delta)^{\alpha/2}\varphi_k(\cot(s))$ was expressed in terms of $s\in(0,\pi)$, for both even and odd values of $k$. More precisely, the following result was proved. Let $k\in\mathbb Z$ and $s\in(0,\pi)$. Then,
\begin{equation}
\label{e:fraclapphiksaone}
(-\Delta)^{1/2}\varphi_k(\cot(s)) = 
	\left\lbrace
	\begin{aligned}
		& |k|\sin^2(s)e^{iks}, & \text{$k$ even,}
		\\
		& {-\frac{ik}{\pi}}\left(\frac{2}{k^2-4} + \sum_{l=-\infty}^\infty\frac{4\sgn(l)e^{i2ls}}{(k-2l)((k-2l)^2-4)}\right), & \text{$k$ odd.}
	\end{aligned}
	\right.
\end{equation}
Moreover, let $\alpha\in(0,1)\cup(1,2)$, $k\in\mathbb Z$ and $s\in(0,\pi)$. Then,
\begin{equation}
\label{e:fraclapphiks}
	(-\Delta)^{\alpha/2}\varphi_k(\cot(s))
	=
	\left\{
	\begin{aligned}
		& \frac{c_\alpha\sin^{\alpha-1}(s)}{8\tan(\alpha\frac\pi2)}\sum_{l=-\infty}^\infty e^{i2ls}((1-\alpha)k^2-4kl)
		\\
		& \quad \times \frac{\Gamma\left(\frac{-1+\alpha}{2}+|l|\right)\Gamma\left(\frac{-1-\alpha}{2}+\left|\frac{k}{2}-l\right|\right)}{\Gamma\left(\frac{3-\alpha}{2}+|l|\right)\Gamma\left(\frac{3+\alpha}{2}+\left|\frac{k}{2}-l\right|\right)}, & \text{$k$ even,}
		\\
		& i\frac{c_\alpha\sin^{\alpha-1}(s)}{8}\sum_{l=-\infty}^\infty e^{i2ls}((1-\alpha)k^2-4kl)
		\\
		&  \quad \times\sgn(\tfrac{k}{2}-l)\frac{\Gamma\left(\frac{-1+\alpha}{2}+|l|\right)\Gamma\left(\frac{-1-\alpha}{2}+\left|\frac{k}{2}-l\right|\right)}{\Gamma\left(\frac{3-\alpha}{2}+|l|\right)\Gamma\left(\frac{3+\alpha}{2}+\left|\frac{k}{2}-l\right|\right)}, & \text{$k$ odd,}
	\end{aligned}
	\right.
\end{equation}
where $c_\alpha$ is given by \eqref{e:ca}. Note that in this paper, we work rather with $(-\Delta)^{\alpha/2}\lambda_k(\cot(s))$, whose expression follows trivially after replacing $k$ by $2k$ in the even case of \eqref{e:fraclapphiksaone} and \eqref{e:fraclapphiks}, because  $\lambda_{k}(x)\equiv\varphi_{2k}(x)$:
\begin{equation}
	\label{e:fraclaplks}
	(-\Delta)^{\alpha/2}\lambda_k(\cot(s))
	=
	\left\{
	\begin{aligned}
		& 2|k|\sin^2(s)e^{i2ks}, & \alpha=1,
		\\
		& \frac{c_\alpha\sin^{\alpha-1}(s)}{2\tan(\alpha\frac\pi2)}\sum_{l=-\infty}^\infty e^{i2ls}((1-\alpha)k^2-2kl)
		\\
		& \quad \times \frac{\Gamma\left(\frac{-1+\alpha}{2}+|l|\right)\Gamma\left(\frac{-1-\alpha}{2}+|k-l|\right)}{\Gamma\left(\frac{3-\alpha}{2}+|l|\right)\Gamma\left(\frac{3+\alpha}{2}+|k-l|\right)}, & \alpha\not=1.
	\end{aligned}
	\right.
\end{equation}
This last formula is very convenient from a numerical point of view, and, thanks to the identities in \eqref{e:dalphalkvsflalk}, we have its corresponding version for $\Da[\lambda_{k}](x)$ and $\barDa[\lambda_{k}](x)$, as the following corollary shows.
\begin{corollary}\label{cor:Dalks}
Let $\lambda_k(x)$ be defined as in \eqref{e:lk}, $\alpha\in(0,1)$, $k\in\mathbb{Z}$ and $s\in(0,\pi)$. Then,
\begin{align*}
		\Da[\lambda_{k}](\cot(s)) \equiv & \frac{e^{-i\sgn(k)\alpha\frac\pi2}c_\alpha\sin^{\alpha-1}(s)}{2\tan(\alpha\frac\pi2)}\sum_{l=-\infty}^\infty e^{i2ls}((1-\alpha)k^2-2kl)
		\cr
		& \quad \times  \frac{\Gamma\left(\frac{-1+\alpha}{2}+|l|\right)\Gamma\left(\frac{-1-\alpha}{2}+|k-l|\right)}{\Gamma\left(\frac{3-\alpha}{2}+|l|\right)\Gamma\left(\frac{3+\alpha}{2}+|k-l|\right)},
		\\
		\barDa[\lambda_{k}](\cot(s)) \equiv & -\frac{e^{i\sgn(k)\alpha\frac\pi2}c_\alpha\sin^{\alpha-1}(s)}{2\tan(\alpha\frac\pi2)}\sum_{l=-\infty}^\infty e^{i2ls}((1-\alpha)k^2-2kl)
		\cr
		& \quad \times  \frac{\Gamma\left(\frac{-1+\alpha}{2}+|l|\right)\Gamma\left(\frac{-1-\alpha}{2}+|k-l|\right)}{\Gamma\left(\frac{3-\alpha}{2}+|l|\right)\Gamma\left(\frac{3+\alpha}{2}+|k-l|\right)}.
	\end{align*}
Likewise, when $\alpha\in(1,2)$, $k\in\mathbb{Z}$ and $s\in(0,\pi)$. Then,
\begin{align*}
	\partial_x\Da[\lambda_{k}](\cot(s)) \equiv & \frac{e^{-i\sgn(k)\alpha\frac\pi2}c_\alpha\sin^{\alpha-1}(s)}{2\tan(\alpha\frac\pi2)} \sum_{l=-\infty}^\infty e^{i2ls}((1-\alpha)k^2-2kl)
	\cr
	& \quad  \times \frac{\Gamma\left(\frac{-1+\alpha}{2}+|l|\right)\Gamma\left(\frac{-1-\alpha}{2}+|k-l|\right)}{\Gamma\left(\frac{3-\alpha}{2}+|l|\right)\Gamma\left(\frac{3+\alpha}{2}+|k-l|\right)},
	\\
	\partial_x\barDa[\lambda_{k}](\cot(s)) \equiv & \frac{e^{i\sgn(k)\alpha\frac\pi2}c_\alpha\sin^{\alpha-1}(s)}{2\tan(\alpha\frac\pi2)}\sum_{l=-\infty}^\infty e^{i2ls}((1-\alpha)k^2-2kl)
	\cr
	& \quad \times  \frac{\Gamma\left(\frac{-1+\alpha}{2}+|l|\right)\Gamma\left(\frac{-1-\alpha}{2}+|k-l|\right)}{\Gamma\left(\frac{3-\alpha}{2}+|l|\right)\Gamma\left(\frac{3+\alpha}{2}+|k-l|\right)}.
\end{align*}

\end{corollary}

\begin{proof}
Trivial, from \eqref{e:fraclaplks} and the identities in \eqref{e:dalphalkvsflalk}.
\end{proof}

\subsection{Discussion on the Weyl-Marchaud derivatives of $\varphi_k(x)$, for $k$ odd}

\label{s:Dphik}

In the previous pages, we have dealt with $\Da[\lambda_{k}]$, $\barDa[\lambda_{k}]$, $\partial_x\Da[\lambda_{k}]$ and $\partial_x\barDa[\lambda_{k}]$, where $\lambda_{k}(x) \equiv\varphi_{2k}(x)$. Therefore, a natural question that arises is whether it is possible to obtain explicit formulas for the Weyl-Marchaud derivatives of $\varphi_k(x)$, for $k$ odd. Unfortunately, those formulas are still not known to us, but we think it is worth discussing where the difficulties lie.

From now we suppose, without loss of generality, that $k > 0$. Recall that, from the definition of $\varphi_k(x)$ in \eqref{e:phik}:
\begin{equation}
\label{e:phikphikp}
\varphi_k(x) = \frac{(x + i)^k}{(1 + x^2)^{k/2}} \Longrightarrow \varphi_k'(x) = -\frac{ik(x + i)^k}{(1 + x^2)^{1+k/2}},
\end{equation}
Then, introducing $\varphi_{k}$ into \eqref{Int:Dalpha}:
\begin{equation}
\label{e:Daphikx}
\Da[\varphi_k](x) = -\frac{1}{\Gamma(1-\alpha)} \int^{\infty}_{0} \frac{1}{z^\alpha}\frac{ik(x - z + i)^k}{(1 + (x - z)^2)^{1+k/2}}dz.
\end{equation}
Likewise, introducing $\varphi_{k}(x)$ into \eqref{Int:Dalpha:bar}:
\begin{equation}
	\label{e:Dabarphikx}
	\barDa[\varphi_k](x) = -\frac{1}{\Gamma(1-\alpha)} \int^{\infty}_{0} \frac{1}{z^\alpha}\frac{ik(x + z + i)^k}{(1 + (x + z)^2)^{1+k/2}}dz.
\end{equation}
When $k$ is even, $k/2$ is an integer, so the denominator in the integrals in \eqref{e:Daphikx} and \eqref{e:Dabarphikx} is a polynomial, and hence, those integrals can be computed by complex variable techniques, as we have done with the equivalent equations \eqref{Dalpha:ln} and \eqref{Dalpha:ln:bar}. However, when $k$ is odd, the denominator in the integrals in \eqref{e:Daphikx} and \eqref{e:Dabarphikx} is no longer a polynomial, but the square root of a polynomial, and the same arguments cannot be applied.

On the other hand, in view of the nontrivial identities in \eqref{e:dalphalkvsflalk} and \eqref{e:dalphamukvsflamuk}, one might be tempted to think that they can be valid for functions other than $\lambda_k(x)$ and $\mu_k(x)$, in particular for $\varphi_k(x)$, with $k$ odd, but unfortunately, this is not true. To illustrate this, we will show in the following pages that \eqref{e:dalphalkvsflalk} and \eqref{e:dalphamukvsflamuk} are no longer valid when applied to $\varphi_k(x)$, for $k$ odd. Indeed, even if we do not still know an explicit expression for the Weyl-Marchaud derivatives of those functions, we are at least able to compute them at $x = 0$. 

Let us compute first $\Da[\varphi_k](0)$. Evaluating \eqref{e:Daphikx} at $x = 0$:
\begin{align*}
	\Da[\varphi_k](0) & = -\frac{ik}{\Gamma(1-\alpha)} \int^{\infty}_{0}\frac{z^{-\alpha}(-z + i)^k}{(1 + z^2)^{1+k/2}}dz
	\cr
	& = -\frac{i^{1+k}k}{\Gamma(1-\alpha)}\sum_{n=0}^{k}i^n\binom{k}{n}\int^{\infty}_{0}\frac{z^{n-\alpha}}{(1 + z^2)^{1+k/2}}dz
	\cr
	& = -\frac{i^{1+k}k}{2\Gamma(1-\alpha)\Gamma\left(1+\frac k2\right)}\sum_{n=0}^{k}i^n\binom{k}{n}\Gamma\left(\frac{1 + \alpha + k - n}2\right)\Gamma\left(\frac{1-\alpha+n}2\right),
\end{align*}
where we have used that
\begin{align*}
	\int^{\infty}_{0}\frac{z^p}{(1 + z^2)^q}dz & = \int^{\pi/2}_{0}\frac{\tan^p(x)}{(1 + \tan^2(x))^q\cos^2(x)}dx = \int^{\pi/2}_{0}\sin^p(x)\cos^{-2-p+2q}(x)dx
	\cr
	& =  \frac12\int_0^1y^{(1+p)/2-1}(1-y)^{(-1-p)/2+q-1}dy = \frac12B\left(\frac{1+p}2,\frac{-1-p}2+q\right)
	\cr
	& = \frac{1}{2\Gamma(q)}\Gamma\left(\frac{1+p}2\right)\Gamma\left(\frac{-1-p}2+q\right),
\end{align*}
with $B(\cdot,\cdot)$ being Euler's beta function, for $p = n - \alpha$, $q = 1 + k/2$. Likewise, in order to compute $\barDa[\varphi_k](0)$, we evaluating \eqref{e:Dabarphikx} at $x = 0$:
\begin{align*}
	\barDa[\varphi_k](0) & = -\frac{ik}{\Gamma(1-\alpha)} \int^{\infty}_{0}\frac{z^{-\alpha}(z + i)^k}{(1 + z^2)^{1+k/2}}dz
	\cr
	& = -\frac{i^{1+k}k}{\Gamma(1-\alpha)}\sum_{n=0}^{k}(-i)^n\binom{k}{n}\int^{\infty}_{0}\frac{z^{n-\alpha}}{(1 + z^2)^{1+k/2}}dz
	\cr
	& = -\frac{i^{1+k}k}{2\Gamma(1-\alpha)\Gamma\left(1+\frac k2\right)}\sum_{n=0}^{k}(-i)^n\binom{k}{n}\Gamma\left(\frac{1 + \alpha + k - n}2\right)\Gamma\left(\frac{1-\alpha+n}2\right).
\end{align*}
Finally, in order to compute $(-\Delta)^{\alpha/2}\varphi_k(x)$ at $x = 0$, we use the following definition of the fractional Laplacian from \cite[(7)]{cayamacuestadelahoz2021}:
\begin{equation*}
(-\Delta)^{\alpha/2}\varphi_k(x) = \frac{c_\alpha}{\alpha}\int^{\infty}_{0} \frac{\varphi_k'(x-z) - \varphi_k'(x+z)}{z^\alpha}dz,
\end{equation*}
with $c_\alpha$ given by \eqref{e:ca}. Hence, reasoning as above,
\begin{align*}
	(-\Delta)^{\alpha/2}\varphi_k(0) & = -\frac{ik2^{\alpha-1}\Gamma\left(\frac{1+\alpha}2\right)}{\sqrt{\pi}\Gamma\left(1-\frac\alpha2\right)} \int^{\infty}_{0} \frac{1}{z^\alpha}\left[\frac{(- z + i)^k}{(1 + z^2)^{1+k/2}} - \frac{(z + i)^k}{(1 + z^2)^{1+k/2}}
	\right]dz
	\cr
	& = \frac{i^{k}2^\alpha\Gamma\left(\frac{1+\alpha}2\right)}{\sqrt{\pi}\Gamma\left(\frac k2\right)\Gamma\left(1-\frac\alpha2\right)}
\sum_{n=0}^{k}\sin\left(n\frac\pi2\right)\binom{k}{n}\Gamma\left(\frac{1 + \alpha + k - n}2\right)\Gamma\left(\frac{1-\alpha+n}2\right),
\end{align*}
where we have used that $(-i)^n - i^n = -2i\sin(n\pi/2)$. We remark that $(-\Delta)^{\alpha/2}\varphi_k(0)$ is pure real for $k$ even, and pure imaginary for $k$ odd, unlike $\Da[\varphi_k](0)$ and $\barDa[\varphi_k](x)$, which have nonzero real and immaginary parts. For instance, when $k = 1$,
\begin{align*}
\Da[\varphi_1](0) & = \frac{\Gamma\left(1 + \frac\alpha2\right)\Gamma\left(\frac{1-\alpha}2\right) + i\Gamma\left(1-\frac\alpha2\right)\Gamma\left(\frac{1 + \alpha}2\right)}{\sqrt\pi\Gamma(1-\alpha)},
	\cr
\barDa[\varphi_1](0) & = \frac{\Gamma\left(1 + \frac\alpha2\right)\Gamma\left(\frac{1-\alpha}2\right) - i\Gamma\left(1-\frac\alpha2\right)\Gamma\left(\frac{1 + \alpha}2\right)}{\sqrt\pi\Gamma(1-\alpha)},
	\cr
(-\Delta)^{\alpha/2}\varphi_1(0) & = \frac{i2^{\alpha}\Gamma^2\left(\frac{1 + \alpha}2\right)}{\pi} = \frac{i}{\cos(\alpha\frac\pi2)}\cdot\frac{\Gamma\left(1-\frac\alpha2\right)\Gamma\left(\frac{1 + \alpha}2\right)}{\sqrt\pi\Gamma(1-\alpha)},
\end{align*}
where we have applied Euler's reflection formula \eqref{Euler:R} for $\omega = (1 + \alpha)/2$, and Legendre's duplication formula \eqref{e:legendre}, for $\omega = (1 - \alpha)/2$. Therefore, we have
\begin{align*}
\operatorname{Im}(\Da[\varphi_1](0)) & = -i\cos\left(\alpha\frac\pi2\right)(-\Delta)^{\alpha/2}\varphi_1(0) \Longrightarrow \Da[\varphi_{1}](x) \not\equiv e^{-i\sgn(1)\alpha\frac\pi2}(-\Delta)^{\alpha/2}\varphi_{1}(x),
	\cr
\operatorname{Im}(\barDa[\varphi_1](0)) & = i\cos\left(\alpha\frac\pi2\right)(-\Delta)^{\alpha/2}\varphi_1(0) \Longrightarrow \overline{\Da[\varphi_{1}](x)} \not\equiv -e^{-i\sgn(1)\alpha\frac\pi2}(-\Delta)^{\alpha/2}\varphi_{1}(x),
\end{align*}
and a similar reasoning can be made for any other odd value of $k$, i.e., the identities \eqref{e:dalphalkvsflalk} and \eqref{e:dalphamukvsflamuk} are not valid for $\varphi_k(x)$ with $k$ odd.

The previous arguments show that, even in the case where we know explicitly the expression of $(-\Delta)^{\alpha/2}u(x)$ for a given function $u(x)$, there is unfortunately no universal formula that allows to express $\Da[u](x)$ and $\barDa[u](x)$ in terms of $(-\Delta)^{\alpha/2}u(x)$; this is valid in particular for $\varphi_{k}(x)$, with $k$ odd. As a consequence, we cannot use \eqref{e:fraclapphiks} with odd $k$, to express $\Da[\varphi_k](\cot(s))$, $\barDa[\varphi_k](\cot(s))$, $\partial_x\Da[\varphi_k](\cot(s))$ and $\partial_x\barDa[\varphi_k](\cot(s))$ in terms of $(-\Delta)^{\alpha/2}\varphi_k(\cot(s))$. On the other hand, we have been unable to deduce the formulas in Corollary \ref{cor:Dalks} directly from \eqref{e:Daphikx} and \eqref{e:Dabarphikx}, without resorting to the identities in \eqref{e:dalphalkvsflalk}; this requires further research, and we postpone it for the future.

\section{Riesz-Feller operators acting on $\lambda_k(x)$}

\label{s:RFln}

In what follows, we obtain explicit expressions for the Riesz-Feller operators acting on $\lambda_k(x)$. Even if most expressions are based on the results in Section \ref{s:Dphik}, we gather them in a separate section, because of their importance.
\begin{theorem}
	
Let $\lambda_k(x)$ be defined as in \eqref{e:lk}, $\alpha\in(0,2)$, $k\in\mathbb Z$, $x\in\mathbb R$. Then,
\begin{align}
\label{e:Daglk}
\Dag[\lambda_{k}](x) & = -e^{i\sgn(k)\gamma\frac\pi2}(-\Delta)^{\alpha/2}\lambda_{k}(x)
	\\
\label{e:DaglkF}
& = \frac{2e^{i\sgn(k)\gamma\frac\pi2}|k|\Gamma(1+\alpha)}{(i\sgn(k)x+1)^{1+\alpha}}{}_2F_1\left(1 - |k|, 1 + \alpha; 2; \frac{2}{i\sgn(k)x+1}\right).
\end{align}
Moreover, when $\alpha = 1$, this expression can be simplified to
\begin{equation}
\label{e:Doneglk}
D^1_\gamma[\lambda_{k}](x) = -\frac{2e^{i\sgn(k)\gamma\frac\pi2}|k|(x+i)^{2k}}{(1+x^2)^{1+k}}.
\end{equation}

\end{theorem}

\begin{proof}

When $\alpha\in(0,1)$, from \eqref{RF:01:integral:representation},
\begin{align*}
\Dag[\lambda_{k}](x) & = \Gamma(-\alpha)\left(c_{\gamma,\alpha}^1\Da[\lambda_{k}](x) - c_{\gamma,\alpha}^2 \barDa[\lambda_{k}](x)\right)
	\cr
& = \Gamma(-\alpha)\left(c_{\gamma,\alpha}^1e^{-i\sgn(k)\alpha\frac\pi2}(-\Delta)^{\alpha/2}\lambda_{k}(x) - c_{\gamma,\alpha}^2 (-e^{i\sgn(k)\alpha\frac\pi2})(-\Delta)^{\alpha/2}\lambda_{k}(x)\right)
	\cr
& = \Gamma(-\alpha)\left(c_{\gamma,\alpha}^1(-i\sgn(k))^\alpha + c_{\gamma,\alpha}^2(i\sgn(k))^\alpha\right)(-\Delta)^{\alpha/2}\lambda_{k}(x)
	\cr
& = -|{-\sgn(k)}|^\alpha e^{-i\sgn(-\sgn(k))\gamma\frac\pi2}(-\Delta)^{\alpha/2}\lambda_{k}(x)
	\cr
& = - e^{i\sgn(k)\gamma\frac\pi2}(-\Delta)^{\alpha/2}\lambda_{k}(x),
\end{align*}
where we have taken $\xi = -\sgn(k)$ in \eqref{check:symbol}. Likewise, when $\alpha\in(1,2)$, from \eqref{RF:integral:representation},
\begin{align*}
\Dag[\lambda_{k}](x) & = \Gamma(-\alpha)\left(c_{\gamma,\alpha}^1\partial_x\DD^{\alpha-1}[u](x) + c_{\gamma,\alpha}^2 \partial_x\overline{\DD^{\alpha-1}}[u](x)\right)
	\cr
& = \Gamma(-\alpha)\left(c_{\gamma,\alpha}^1e^{-i\sgn(k)\alpha\frac\pi2}(-\Delta)^{\alpha/2}\lambda_{k}(x) + c_{\gamma,\alpha}^2e^{i\sgn(k)\alpha\frac\pi2}(-\Delta)^{\alpha/2}\lambda_{k}(x)\right)
	\cr
& = -e^{i\sgn(k)\gamma\frac\pi2}(-\Delta)^{\alpha/2}\lambda_{k}(x).
\end{align*}
Finally, when $\alpha = 1$, from \eqref{e:D1g} and \eqref{e:fraclapa1},
\begin{align}
\label{e:Doneglkproof}
D^1_\gamma[\lambda_k](x) & = -\cos \left(\gamma\frac\pi2\right)(-\Delta)^{1/2}\lambda_k(x) +\sin \left(\gamma\frac\pi2\right)\lambda_k'(x)
	\cr
& = -\cos \left(\gamma\frac\pi2\right)(-\Delta)^{1/2}\lambda_k(x) - i\sgn(k)\sin\left(\gamma\frac\pi2\right)(-\Delta)^{1/2}\lambda_k(x)
	\cr
& = -e^{i\sgn(k)\gamma\frac\pi2}(-\Delta)^{1/2}\lambda_k(x),
\end{align}
which concludes the proof of \eqref{e:Daglk}. Then, \eqref{e:DaglkF} follows after applying \eqref{e:fraclap2F1} to \eqref{e:Daglk}, and \eqref{e:Doneglk} follows after applying \eqref{e:fraclapa1} to \eqref{e:Doneglkproof}.

\end{proof}

Note that, from \eqref{e:fraclapphiks} and \eqref{e:DaglkF}, it is straightforward to express $\Dag[\lambda_{k}](\cot(s))$ in terms of $s$, as shows the following result.
\begin{corollary}

\label{cor:Daglks}

Let $\lambda_k(x)$ be defined as in \eqref{e:lk}, $\alpha\in(0,1)\cup(1,2)$, $k\in\mathbb{Z}$ and $s\in(0,\pi)$. Then,
\begin{align*}
	\Dag[\lambda_{k}](\cot(s)) & = -\frac{c_\alpha e^{i\sgn(k)\gamma\frac\pi2}\sin^{\alpha-1}(s)}{2\tan(\alpha\frac\pi2)}\sum_{l=-\infty}^\infty e^{i2ls}((1-\alpha)k^2-2kl)
	\\
	& \qquad \times \frac{\Gamma\left(\frac{-1+\alpha}{2}+|l|\right)\Gamma\left(\frac{-1-\alpha}{2}+|k-l|\right)}{\Gamma\left(\frac{3-\alpha}{2}+|l|\right)\Gamma\left(\frac{3+\alpha}{2}+|k-l|\right)}.
\end{align*}
Moreover, when $\alpha = 1$, $k\in\mathbb{Z}$ and $s\in(0,\pi)$,
\begin{equation}
\label{e:Doneglks}
D^1_\gamma[\lambda_{k}](\cot(s)) = -2e^{i\sgn(k)\gamma\frac\pi2}|k|\sin^2(s)e^{i2ks}.
\end{equation}
\end{corollary}

\begin{proof}
Trivial, from \eqref{e:fraclaplks} and \eqref{e:Daglk}. The $\alpha = 1$ case also follows after evaluating \eqref{e:Doneglk} at $x = \cot(s)$.

\end{proof}

As we have already said, when $k$ is odd, and $\alpha\in(0,1)\cup(1,2)$, we still do not know explicit expressions for $\Da[\varphi_k](x)$, $\barDa[\varphi_k](x)$, $\partial_x\Da[\varphi_k](x)$, $\partial_x\barDa[\varphi_k](x)$, so we do not have explicit expressions for $\Dag[\varphi_k](x)$ with $k$ odd either. However, in the special case $\alpha = 1$, the computation of $D^1_\gamma[\varphi_k](x)$ can be done by means of \eqref{e:D1g}, because $(-\Delta)^{1/2}\varphi_k(x)$ is known for $k$ odd \cite[Th. 2.4]{CuestaDelaHozGirona2023}). More precisely, if $k$ is an odd integer, and $x\in\mathbb R$, we have
\begin{align}
\label{e:fraclapphika}
(-\Delta)^{1/2} \varphi_k(x) & = {-\frac{2i}{\pi (2+k)}} - \frac{k(x + i)^k}{(1+x^2)^{1+k/2}}
+ \frac{8i}{\pi (4-k^2)}{}_2F_1\left(1, -1 - \frac k2; 2 - \frac k2; \frac{(x+i)^2}{1+x^2} \right).
\end{align}
Moreover, even if a direct application of \eqref{e:2F1} would yield an infinite sum, it is possible to compute ${}_2F_1(1, -1 - k/2; 2 - k/2; (x+i)^2/(1+x^2))$ by summing a finite number of terms (see \cite{CuestaDelaHozGirona2023}), from which the following equivalent expression follows:
\begin{align}
	\label{e:fraclapphikb}
	(-\Delta)^{1/2} \varphi_k(x) & = {-\frac{2i\sgn(k)}{\pi (2+|k|)}} - \frac{2ik(x + i)^k[x(1+x^2)^{1/2} + \arg\sinh(x)]}{\pi(1+x^2)^{1+k/2}}
	\cr
& \qquad - \sum_{n=0}^{(|k|-1)/2}\frac{8ik(x + i)^{[k-\sgn(k)(2n+1)]}}{\pi(2n-1)(2n+1)(2n+3)(1+x^2)^{[k-\sgn(k)(2n+1)]/2}}.
\end{align}
On the other hand, \eqref{e:fraclapphika} and \eqref{e:fraclapphikb} can be written in a slightly more compact way, when evaluated at $x = \cot(s)$, with $s\in(0,\pi)$, and expressed in terms of $s$:
\begin{align}
\label{e:fraclapphiksa}
(-\Delta)^{1/2}\varphi_k(\cot(s)) & = {-\frac{2i}{\pi (2+k)}} - k\sin^2(s)e^{iks}
 + \frac{8i}{\pi (4-k^2)}{}_2F_1\left(1, -1 - \frac k2; 2 - \frac k2; e^{i2s}\right);
\end{align}
or equivalently,
\begin{align}
\label{e:fraclapphiksb}
(-\Delta)^{1/2}\varphi_k(\cot(s)) & = {-\frac{2i\sgn(k)}{\pi (2+|k|)}} - \frac{2ike^{iks}}{\pi}\Bigg[\cos(s) + \sin^2(s)\ln\left(\cot\left(\frac s2\right)\right)
		\cr
& \qquad + \sum_{n=0}^{(|k|-1)/2}\frac{4e^{-i\sgn(k)(2n+1)s}}{(2n-1)(2n+1)(2n+3)}\Bigg].
\end{align}
Bearing in mind \eqref{e:fraclapphika} and \eqref{e:fraclapphiksa}, we have the following result.
\begin{lemma}

Let $k$ be an odd integer, $\varphi_k(x)$ be defined as in \eqref{e:phik}, $|\gamma|\le 1$, and $x\in\mathbb R$. Then,
\begin{align}
\label{e:Donegphika}
D^1_\gamma[\varphi_k](x) & = \frac{ke^{-i\gamma\frac\pi2}(x + i)^k}{(1+x^2)^{1+k/2}} +\cos\left(\gamma\frac\pi2\right)\bigg[{\frac{2i}{\pi (2+k)}} 
	\cr
	& \qquad - \frac{8i}{\pi (4-k^2)}{}_2F_1\left(1, -1 - \frac k2; 2 - \frac k2; \frac{(x+i)^2}{1+x^2} \right)\bigg],
\end{align}
or, equivalently,
\begin{align*}
D^1_\gamma[\varphi_k](x) & = \cos \left(\gamma\frac\pi2\right)\Bigg[{\frac{2i\sgn(k)}{\pi (2+|k|)}} + \frac{2ik(x + i)^k[x(1+x^2)^{1/2} + \arg\sinh(x)]}{\pi(1+x^2)^{1+k/2}}
	\cr
& \qquad + \sum_{n=0}^{(|k|-1)/2}\frac{8ik(x + i)^{[k-\sgn(k)(2n+1)]}}{\pi(2n-1)(2n+1)(2n+3)(1+x^2)^{[k-\sgn(k)(2n+1)]/2}}\Bigg]
	\cr
& \qquad - \sin \left(\gamma\frac\pi2\right) \frac{ik(x + i)^k}{(1 + x^2)^{1+k/2}}.
\end{align*}
Moreover, if $s\in(0,\pi)$, then
\begin{align*}
D^1_\gamma[\varphi_k](\cot(s)) & = ke^{-i\gamma\frac\pi2}\sin^2(s)e^{iks}
	\cr
	& \qquad +\cos\left(\gamma\frac\pi2\right)\bigg[{\frac{2i}{\pi (2+k)}} - \frac{8i}{\pi (4-k^2)}{}_2F_1\left(1, -1 - \frac k2; 2 - \frac k2; e^{i2s} \right)\bigg],
\end{align*}
or, equivalently,
\begin{align}
\label{e:Donegphiksb}
D^1_\gamma[\varphi_k](\cot(s)) & = \cos \left(\gamma\frac\pi2\right)\Bigg[{\frac{2i\sgn(k)}{\pi (2+|k|)}} + \frac{2ike^{iks}}{\pi}\Bigg[\cos(s) + \sin^2(s)\ln\left(\cot\left(\frac s2\right)\right)
	\cr
& \qquad + \sum_{n=0}^{(|k|-1)/2}\frac{4e^{-i\sgn(k)(2n+1)s}}{(2n-1)(2n+1)(2n+3)}\Bigg]\Bigg] - \sin \left(\gamma\frac\pi2\right) ik\sin^2(s)e^{iks}.
\end{align}

\end{lemma}

\begin{proof}

We apply  \eqref{e:D1g}, for $u(x) = \varphi_k(x)$, to \eqref{e:fraclapphika}, \eqref{e:fraclapphikb}, \eqref{e:fraclapphiksa} and \eqref{e:fraclapphiksb}, respectively. $\varphi'(x)$ is given by \eqref{e:phikphikp}, and, evaluating 
\eqref{e:phikphikp} at $x = \cot(s)$, we get $\varphi_k'(\cot(s)) = -ik\sin^2(s)e^{iks}$.

\end{proof}

To finish this section, we offer, for the sake of completeness, explicit expressions of $\Dag$ applied to the complex Christov functions $\mu_k(x)$.

\begin{corollary}
	
Let $\mu_k(x)$ be defined as in \eqref{e:muk},  $\alpha\in(0,2)$, $k\in\mathbb{Z}$ and $x\in\mathbb{R}$. Then,
\begin{equation*}
\Dag[\mu_{k}](x) \equiv -e^{i\sgn(k)\gamma\frac\pi2}(-\Delta)^{\alpha/2}\mu_{k}(x),
\end{equation*}
where $(-\Delta)^{\alpha/2}\mu_k(x)$ is given by \eqref{e:fraclapmuk}.

\end{corollary}

\begin{proof}
Identical to that of Corollary \ref{cor:Damuk}:
	\begin{align*}
		\Dag[\mu_{k}](x) & = \Dag\left[\frac{\lambda_k - \lambda_{k+1}}{2}\right](x) = \frac{\Dag[\lambda_{k}](x) - \Dag[\lambda_{k+1}](x)}{2}
		\cr
		& = -\frac{e^{i\sgn(k)\gamma\frac\pi2}(-\Delta)^{\alpha/2}\lambda_{k}(x) - e^{-i\sgn(k+1)\gamma\frac\pi2}(-\Delta)^{\alpha/2}\lambda_{k+1}(x)}{2}.
		\cr
		& = -e^{i\sgn(k)\gamma\frac\pi2}(-\Delta)^{\alpha/2}\frac{\lambda_{k}(x) - \lambda_{k+1}(x)}{2} = -e^{i\sgn(k)\gamma\frac\pi2}(-\Delta)^{\alpha/2}\mu_k(x),
	\end{align*}
	where we have used again that, when $k \not = 0$ and $k \not = -1$, $\sgn(k) \equiv \sgn(k+1)$, when $k = 0$, $(-\Delta)^{\alpha/2}\lambda_{k}(x) = (-\Delta)^{\alpha/2}1 = 0$, and, when $k = -1$, $(-\Delta)^{\alpha/2}\lambda_{k+1}(x) = (-\Delta)^{\alpha/2}1 = 0$.
	
\end{proof}

\section{Numerical experiments}

\label{s:numerical}

\subsection{A pseudospectral approach}

\label{s:PS}

Given a continuous bounded function $u(x)$, such that its limits as $x\to\pm\infty$ are the same, we perform the change of variable $x = L\cot(s)$, where $s\in(0,\pi)$, and $L > 0$ is a scaling factor, and define $U(s) \equiv u(L\cot(s))$, which maps the unbounded domain $\mathbb R$ into a bounded one, hence avoiding truncation of the domain. Then, we expand $U(s)$, which we regard as $\pi$-periodic, as a Fourier series:
\begin{equation*}
U(s) = \sum_{k = -\infty}^{\infty}\hat u(k)e^{i2ks},
\end{equation*}
or, in terms of $x$,
\begin{equation*}
u(x) = U\left(\arccot\left(\frac xL\right)\right) = \sum_{k = -\infty}^{\infty}\hat u(k)e^{i2k\arccot(x/L)} = \sum_{k = -\infty}^{\infty}\hat u(k)\lambda_k\left(\frac xL\right),
\end{equation*}
where the multivalued function $\arccot(x)$ is defined in such a way that it takes values on $(0,\pi)$; note that we can evaluate it numerically by doing $\arccot(x) = \pi/2 - \arctan(x)$. Moreover, introducing the notation
\begin{equation*}
\lambda_{k,L}(x) \equiv \lambda_k(x/L)  \equiv \left(\frac{ix - L}{ix + L}\right)^{k} \equiv \frac{(x + iL)^{2k}}{(L^2 + x^2)^k}, \quad k\in\mathbb Z,
\end{equation*}
we can write
\begin{equation}
\label{e:seriesfx}
u(x)  = \sum_{k = -\infty}^{\infty}\hat u(k)\lambda_{k,L}(x).
\end{equation}
In this paper, we follow a pseudospectral approach, i.e., we approximate \eqref{e:seriesfx} as a finite sum:
\begin{equation}
\label{e:seriesfxPS}
u(x)  \approx \sum_{k = -\lfloor N/2\rfloor}^{\lceil N/2\rceil-1}\hat u_k\lambda_{k,L}(x),
\end{equation}
where $\lfloor\cdot\rfloor$ and $\lceil\cdot\rceil$ denote respectively the floor and ceiling functions, and the Fourier coefficients $\hat u_k$ of the pseudospectral approximation, which are not identical to the Fourier coefficients $\hat u(k)$ in \eqref{e:seriesfx}, are determined by imposing the equality at $N$ given nodes $x_j = L\cot(s_j)$:
$$
u(x_j) \equiv \sum_{k = -\lfloor N/2\rfloor}^{\lceil N/2\rceil-1}\hat u_k\lambda_{k,L}(x_j) \Longleftrightarrow U(s_j) \equiv \sum_{k = -\lfloor N/2\rfloor}^{\lceil N/2\rceil -1}\hat u_k\lambda_{k,L}(L\cot(s_j)) \equiv \sum_{k = -\lfloor N/2\rfloor}^{\lceil N/2\rceil-1}\hat u_ke^{i2ks_j},
$$
where we have used that $\lambda_{k,L}(L\cot(s)) \equiv \lambda_k(\cot(s)) \equiv e^{i2ks}$. Hence, choosing
\begin{equation}
\label{e:nodessj}
s_j = \frac{\pi(2j+1)}{2N},
\end{equation}
for $j\in\{0, \ldots, N-1\}$, we have that
\begin{align*}
	U(s_j) & = \sum_{k = -\lfloor N/2\rfloor}^{\lceil N/2\rceil-1}\left[e^{i\pi k/N}\hat u_k\right]e^{i2\pi jk/N}
	\cr
	& = \sum_{k = 0}^{\lceil N/2\rceil-1}\left[e^{i\pi k/N}\hat u_k\right]e^{i2\pi jk/N} + \sum_{k = \lceil N/2\rceil}^{N-1}\left[e^{i\pi (k-N)/N}\hat u_k\right]e^{i2\pi jk/N},
\end{align*}
so the values $\{U(s_j)\}$ are given by the inverse discrete Fourier transform (IDFT) of 
\begin{equation*}
	\Big\{N\hat u_0, Ne^{i\pi/N}\hat u_1, \ldots, Ne^{i\pi(\lceil N/2\rceil-1)/N}\hat u_{\lceil N/2\rceil-1},
	Ne^{-i\pi\lceil N/2\rceil/N}\hat u_{-\lceil N/2\rceil}, \ldots, Ne^{-i\pi/N}\hat u_{-1}\Big\};
\end{equation*}
and, conversely, 
\begin{equation}
\label{e:DFT}
\hat u_k = \frac{e^{-i\pi k/N}}{N}\sum_{j = 0}^{N-1}U(s_j)e^{-i2\pi jk/N}, \quad k\in\{-\lfloor N/2\rfloor, \ldots, \lceil N/2\rceil - 1\},
\end{equation}
i.e., the Fourier coefficients $\{\hat u_k\}$ are given by the discrete Fourier transform (DFT) of $\{U(s_j)\}$, multiplying the result by $e^{-i\pi k/N}/N$, for $k\in\{-\lfloor N/2\rfloor, \ldots, \lceil N/2\rceil - 1\}$. The DFT and IDFT can be computed very efficiently \cite{FFT} by means of the fast Fourier transform (FFT) and of the inverse fast Fourier transform (IFFT), respectively. Moreover, as is usual, we apply a Krasny filter \cite{krasny} to the Fourier coefficents $\hat u_k$, by imposing $\hat u_k \equiv 0$ to those $\hat u_k$, such that $|\hat u_k| \le \varepsilon\ll1$, where $\varepsilon$ is an infinitesimally small positive real value; in this paper, we have taken $\varepsilon = 2^{-52} = 2.2204\ldots\times10^{-16}$.

At this point, we apply $\Da$, $\barDa$, $\partial_x\Db$, $\partial_x\barDb$, and $\Dag$ to \eqref{e:seriesfxPS}, and the problem is reduced respectively to the computation of $\Da[\lambda_{k,L}](x)$, $\barDa[\lambda_{k,L}](x)$, $\partial_x\Db[\lambda_{k,L}](x)$, $\partial_x\barDb[\lambda_{k,L}](x)$, and $\Dag[\lambda_{k,L}](x)$. Furthermore, it is easy to check that
\begin{equation}
\label{e:lklvslka}
\begin{aligned}
\Da[\lambda_{k,L}](L\cot(s)) & \equiv \frac{1}{L^\alpha}\Da[\lambda_k](\cot(s)), \quad\alpha\in(0,1),
	\cr
\barDa[\lambda_{k,L}](L\cot(s)) & \equiv \frac{1}{L^\alpha}\barDa[\lambda_k](\cot(s)), \quad\alpha\in(0,1),
	\cr
\partial_x\Db[\lambda_{k,L}](L\cot(s)) & \equiv \frac{1}{L^\alpha}\Da[\lambda_k](\cot(s)), \quad\alpha\in(1,2),
	\cr
\partial_x\barDb[\lambda_{k,L}](L\cot(s)) & \equiv \frac{1}{L^\alpha}\barDa[\lambda_k](\cot(s)), \quad\alpha\in(1,2),
	\cr
\Dag[\lambda_{k,L}](L\cot(s)) & \equiv \frac{1}{L^\alpha}\Dag[\lambda_k](\cot(s)), \quad\alpha\in(0,2).
\end{aligned}
\end{equation}
Therefore, from \eqref{e:seriesfxPS}, \eqref{e:lklvslka}, \eqref{e:dalphalkvsflalk} and \eqref{e:Daglk}, we conclude that
\begin{equation}
	\label{e:lklvslk}
	\begin{aligned}
		\Da[u](L\cot(s)) & \approx \frac{1}{L^\alpha}\sum_{k = -\lfloor N/2\rfloor}^{\lceil N/2\rceil-1}\hat u_ke^{-i\sgn(k)\alpha\frac\pi2}(-\Delta)^{\alpha/2}\lambda_{k}(\cot(s)), \quad\alpha\in(0,1),
		\cr
		\barDa[u](L\cot(s)) & \approx -\frac{1}{L^\alpha}\sum_{k = -\lfloor N/2\rfloor}^{\lceil N/2\rceil-1}\hat u_ke^{i\sgn(k)\alpha\frac\pi2}(-\Delta)^{\alpha/2}\lambda_{k}(\cot(s)), \quad\alpha\in(0,1),
		\cr
		\partial_x\Db[u](L\cot(s)) & \approx \frac{1}{L^\alpha}\sum_{k = -\lfloor N/2\rfloor}^{\lceil N/2\rceil-1}\hat u_ke^{-i\sgn(k)\alpha\frac\pi2}(-\Delta)^{\alpha/2}\lambda_{k}(\cot(s)), \quad\alpha\in(1,2),
		\cr
		\partial_x\barDb[u](L\cot(s)) & \approx \frac{1}{L^\alpha}\sum_{k = -\lfloor N/2\rfloor}^{\lceil N/2\rceil-1}\hat u_ke^{i\sgn(k)\alpha\frac\pi2}(-\Delta)^{\alpha/2}\lambda_{k}(\cot(s)), \quad\alpha\in(1,2),
		\cr
		\Dag[u](L\cot(s)) & \approx -\frac{1}{L^\alpha}\sum_{k = -\lfloor N/2\rfloor}^{\lceil N/2\rceil-1}\hat u_ke^{i\sgn(k)\gamma\frac\pi2}(-\Delta)^{\alpha/2}\lambda_{k}(\cot(s)), \quad\alpha\in(0,2).
	\end{aligned}
\end{equation}
Observe that the expressions for $\Da[u](L\cot(s))$ and $\partial_x\Db[u](L\cot(s))$ are identical, whereas the expressions for $\barDa[u](L\cot(s))$ and $\partial_x\barDb[u](L\cot(s))$ differ only in the sign.

\subsection{Numerical implementation}

As we have seen in \eqref{e:lklvslk}, the numerical approximations we are interested in are reduced to the computation of the fractional Laplacian of $\lambda_k(x)$. One option is to evaluate the corresponding instances of ${}_2F_1$ (see \eqref{e:fraclap2F1}), for which we can use the techniques in \cite{cayamacuestadelahoz2020} involving the use of variable precision arithmetic (vpa). In this regard, let us remember that a direct evaluation of ${}_2F_1$ in software packages like Matlab and Mathematica \cite{mathematica} may yield erroneous results, whereas using vpa avoids this problem, but at the expense of a high computational cost.

On the other hand, in the case with $\alpha = 1$, which is only applicable in the case of Riesz-Feller operators, we have from \eqref{e:lklvslk} and \eqref{e:fraclaplks} that
$$
D_\gamma^1[u](L\cot(s)) \approx -\frac{2\sin^2(s)}{L}\sum_{k = -\lfloor N/2\rfloor}^{\lceil N/2\rceil-1}\hat u_ke^{i\sgn(k)\gamma\frac\pi2}|k|e^{i2ks},
$$
whose implementation is straightforward. Moreover, unlike in the $\alpha\not=1$ case, we do know now an explicit formula also for $D^1_\gamma[\varphi_k](x)$ with $k$ odd, namely \eqref{e:Donegphika}. Therefore, this enables us to approximate numerically $D^1_\gamma[u](x)$, for $u$ such that $\lim_{x\to-\infty}u(x) \not= \lim_{x\to\infty}u(x)$. In that case, $U(s) \equiv u(L\cot(s))$ cannot be regarded as being periodic of period $\pi$, so we must extend it to $s\in[0,2\pi]$ in such a way that $U(s)$ is periodic of period $2\pi$. Then,
\begin{equation*}
U(s) = \sum_{k = -\infty}^{\infty}\hat u(k)e^{iks} \Longrightarrow u(x) = U\left(\arccot\left(\frac xL\right)\right) = \sum_{k = -\infty}^{\infty}\hat u(k)\varphi_k\left(\frac xL\right),
\end{equation*}
and, reasoning as in Section \ref{s:PS},
\begin{equation}
\label{e:Dgonef}
D_\gamma^1[u](L\cot(s)) \approx \frac1L \sum_{k = -N}^{N-1}\hat u_kD_\gamma^1[\varphi_k](\cot(s)),
\end{equation}
where the $2N$ Fourier coefficients $\hat u_k$ are determined by imposing 
$$
U(s_j) \equiv \sum_{k = -N}^{N-1}\hat u_ke^{iks_j};
$$
with $s_j$ being given again by \eqref{e:nodessj}, but taking now $j\in\{0, \ldots, 2N-1\}$. Then, in order to implement \eqref{e:Dgonef}, we consider separately the even and odd values of $k$. In the $k$ even case, bearing in mind that $\varphi_k(x)\equiv\lambda_{k/2}(x)$, we have from \eqref{e:Doneglks} that
$$
D^1_\gamma[\varphi_{k}](\cot(s)) = -e^{i\sgn(k)\gamma\frac\pi2}|k|\sin^2(s)e^{iks}, \quad \text{$k$ even},
$$
and in the $k$ odd case, it is important to remark that, although \eqref{e:Donegphika} does contain an instance of ${}_2F_1$ whose definition involves an a priori infinitely sum, it can be reformulated as an expression containing a finite number of terms \cite{CuestaDelaHozGirona2023}, getting the equivalent expression \eqref{e:Donegphiksb} for $D^1_\gamma[\varphi_k](\cot(s))$. This fact allows to apply a fast convolution algorithm in \eqref{e:Dgonef}, yielding a very fast method that allows considering very large values of $N$. Since, it is straightforward to adapt the techniques in \cite{CuestaDelaHozGirona2023} to compute \eqref{e:Dgonef}, we refer to  \cite{CuestaDelaHozGirona2023} for the details.

In this paper, we will focus on the $\alpha \not= 1$ case, and will use \eqref{e:fraclaplks}, to approximate numerically \eqref{e:lklvslk}. As in \cite{cayamacuestadelahoz2021}, the idea is to construct an operational matrix $\Ma\in\mathcal M_{N\times N}(\mathbb C)$, such that
\begin{equation}
	\label{e:Ma}
	\begin{pmatrix}
		(-\Delta)^{\alpha/2} u(s_0) \\ \vdots \\ (-\Delta)^{\alpha/2} u(s_{N-1})
	\end{pmatrix}
	\approx \Ma\cdot
	\begin{pmatrix}
		\hat u_0 \\ \vdots \\ \hat u_{\lceil N/2\rceil -1} \\ \hat u_{-\lfloor N/2\rfloor} \\ \vdots \\ \hat u_{-1}
	\end{pmatrix},
\end{equation}
but, unlike in \cite{cayamacuestadelahoz2021}, we will consider only functions $u$ such that $U(s) \equiv u(L\cot(s))$ can be regarded as periodic, and we will not extend $U(s)$ to $s\in[0,2\pi]$. Therefore, $\Ma$ is of size $N\times N$ and not $(2N)\times(2N)$, as in \cite{cayamacuestadelahoz2021}. More precisely, the first $\lceil N/2\rceil$ columns of $\Ma$ are $(-\Delta)^{\alpha/2}\lambda_0(s_j), \ldots, (-\Delta)^{\alpha/2}\lambda_{\lceil N/2\rceil-1}(s_j)$, and the last $\lfloor N/2\rfloor$ ones, $(-\Delta)^{\alpha/2}\lambda_{-\lceil N/2\rceil}(s_j), \ldots, (-\Delta)^{\alpha/2}\lambda_{-1}(s_j)$, in that order, for $j\in\{0, \ldots, N-1\}$. Exceptionally, when $N$ is even, it is possible to assigned to the $N/2+1$th column either $(-\Delta)^{\alpha/2}\lambda_{-N/2}(s_j)$ or $(-\Delta)^{\alpha/2}\lambda_{N/2}(s_j)$, because in \eqref{e:Ma}, that column is being multiplied by $\hat u_{-N/2}$, and, from \eqref{e:DFT}, $\hat u_{-N/2}\equiv \hat u_{N/2}$; therefore, in order to avoid the choice, we have taken the $N/2+1$th column to be an all-zero vector, when $N$ is even. We also observe that, since $(-\Delta)^{\alpha/2}\lambda_0(s_j) = (-\Delta)^{\alpha/2}1 = 0$, the first column is always an all-zero vector. Moreover, from \eqref{e:fraclap2F1}, $(-\Delta)^{\alpha/2}\lambda_{-k}(\cot(s_j)) \equiv \overline{(-\Delta)^{\alpha/2}\lambda_k(\cot(s_j))}$, so, in order to generate $\Ma$, it is enough to compute $(-\Delta)^{\alpha/2}\lambda_{k}(\cot(s_j))$ for $k \in\{1, \ldots, \lceil N/2\rceil-1\}$, and obtain $(-\Delta)^{\alpha/2}\lambda_{k}(\cot(s_j))$ for $k \in\{-\lceil N/2\rceil+1, \ldots, -1\}$ by conjugation.

In what follows, we explain briefly how to construct $\Ma$, omitting some of the details, which can be consulted in \cite{cayamacuestadelahoz2021}. More precisely, we evaluate \eqref{e:fraclaplks}, for $\alpha\not=1$, at the nodes $s_j$ given by \eqref{e:nodessj}. Then, we observe that, given a number $l\in\mathbb Z$, we can write it as $l = l_1N + l_2$, where $l_1\in\mathbb Z$, and $l_2\in\{-\lfloor N/2\rfloor, \ldots, \lceil N/2\rceil-1\}$. Hence, 
\begin{equation}
\label{e:aliasing}
e^{i2ls_j} = e^{i2(l_1N+l_2)\pi(2j+1)/(2N)} = (-1)^{l_1}e^{i2l_2s_j},
\end{equation}
i.e., aliasing occurs when evaluating $e^{i2ls}$ at the nodes $s_j$. Therefore, replacing $l$ by $l_1N + l_2$, introducing a double sum in \eqref{e:fraclaplks}, truncating $l_1$ by taking $l_1\in\{-l_{lim}, \ldots, l_{lim}\}$, and bearing in mind \eqref{e:aliasing}, we get the following approximation, for $k\in\{1, \ldots, \lceil N/2\rceil-1\}$:
\begin{align}
	\label{e:thanot1truncated}
	(-\Delta)^{\alpha/2}\lambda_k(\cot(s_j)) & \approx \frac{c_{\alpha}\sin^{\alpha-1}(s_j)}{2\tan(\alpha\frac{\pi}{2})}\sum_{l_2=-\lfloor N/2\rfloor}^{\lceil N/2\rceil-1}\Bigg[
	\sum_{l_1=-l_{lim}}^{l_{lim}}(-1)^{l_1}v_1(l_1N+l_2)
	\cr
	& \qquad \times ((1-\alpha)k^2-2k(l_1N+l_2))v_2(k-l_1N-l_2)\Bigg]e^{i2l_2s_j},
\end{align}
where the functions $v_1(\cdot)$ and $v_2(\cdot)$ are defined for $p\in\mathbb N\cup\{0\}$ as follows:
\begin{align*}
	v_1(p) & \equiv \frac{\Gamma\left(\frac{-1+\alpha}{2}+p\right)}{\Gamma\left(\frac{3-\alpha}{2}+p\right)} \equiv \frac{\Gamma\left(\frac{-1+\alpha}{2}\right)}{\Gamma\left(\frac{3-\alpha}{2}\right)}\prod_{m=0}^{p-1}\frac{\frac{-1+\alpha}{2}+m}{\frac{3-\alpha}{2}+m},
	\cr
	v_2(p) & \equiv \frac{\Gamma\left(\frac{-1-\alpha}{2}+p\right)}{\Gamma\left(\frac{3+\alpha}{2}+p\right)} \equiv \frac{\Gamma\left(\frac{-1-\alpha}{2}\right)}{\Gamma\left(\frac{3+\alpha}{2}\right)}\prod_{m=0}^{p-1}\frac{\frac{-1-\alpha}{2}+m}{\frac{3+\alpha}{2}+m}.
\end{align*}
Note that the definitions of $v_1(p)$ and $v_2(p)$ containing the product sign follow from applying $\Gamma(1+z)=z\Gamma(z)$ recursively. In fact, in order to compute  \eqref{e:thanot1truncated} efficiently, we precompute $v_1(p)$ and $v_2(p)$ recursively starting for $p$, and at least for $0 \le p \le l_{lim}N + \lfloor N/2\rfloor$ and $0 \le p \le l_{lim}N + N - 1$, respectively, and store those values in corresponding arrays.

On the other hand, we only need to compute \eqref{e:thanot1truncated} for $j \in\{0, \ldots, \lceil N/2\rceil-1\}$, and extend the results to the values $j \in\{\lceil N/2\rceil, \ldots, N-1\}$, by using the fact that $e^{i2l_2s_{N-1-j}} \equiv e^{-i2l_2s_j}$, i.e., the last row of $\Ma$ is the conjugate of the first one, the last but one row, the conjugate of the second one, and so on.

The matrix $\Ma$ constructed in this way is equivalent to the matrix $\Ma$ developed in \cite{cayamacuestadelahoz2021} for $L = 1$, after removing the last $N/2$ rows and the columns corresponding to odd values of $k$, so the remarks on the convergence of \eqref{e:thanot1truncated}, and on the choice of $l_{lim}$ apply here, too.

Finally, in order to approximate numerically the formulas in \eqref{e:lklvslk}, we divide all the entries of $\Ma$ by $L^\alpha$, and multiply the columns of $\Ma$ corresponding to positive values of $k$ by $e^{-i\alpha\pi/2}$, $-e^{i\alpha\pi/2}$, $e^{-i\alpha\pi/2}$, $e^{i\alpha\pi/2}$ and $-e^{i\gamma\pi/2}$, respectively, and those corresponding to negative values of $k$, by $e^{i\alpha\pi/2}$, $-e^{-i\alpha\pi/2}$, $e^{i\alpha\pi/2}$, $e^{-i\alpha\pi/2}$ and $-e^{-i\gamma\pi/2}$, respectively.

In Listing \ref{code:fraclapMa}, we offer a full Matlab code, \verb|fraclapMa.m|, containing the function \verb|fraclapMa| that generates $\Ma$. This function has been created by following exactly the steps explained in this section, and, for the sake of completeness, we have added the $\alpha=1$ case, which is reduced to one single extra line based on \eqref{e:fraclaplks} for $\alpha = 1$. Moreover, in Section \ref{s:numerical}, we offer a Matlab example to illustrate how to invoke \verb|fraclapMa| to approximate numerically \eqref{e:lklvslk}.

\begin{lstlisting}[label=code:fraclapMa, language=Matlab, basicstyle=\footnotesize, caption = {Computation of the matrix $\M_\alpha$, when $U(s)$ is not necessarily periodic of period $\pi$}]
function Ma=fraclapMa(N,a,llim) % The parameters are N, alpha and l_{lim}.
k=1:ceil(N/2)-1; % Frequencies k
sj=pi*(2*(0:ceil(N/2)-1)'+1)/(2*N); % s_j
Ma=zeros(N,N); % M_alpha
if a==1 % if alpha is equal to 1:
    Ma(1:ceil(N/2),k+1)=(2*sin(sj).^2*k).*exp(2i*sj*k);
else % if alpha is not equal to 1:
    ca=a*2^(a-1)*gamma(1/2+a/2)/(sqrt(pi)*gamma(1-a/2)); % c_alpha
    l1=(-llim:llim)'; % l_1
    Nl1=N*l1; % N times l_1
    minus1l1=(-1).^(-llim:llim); % (-1)^(l_1)
    [K,NL1]=meshgrid(k,Nl1);
    len1=N*llim+floor(N/2)+1; % Length of the auxiliary array v1
    v1=cumprod([gamma((-1+a)/2),(((-1+a)/2+(0:len1-2)))] ...
        ./[gamma((3-a)/2),((3-a)/2+(0:len1-2))]); % Store v_1(p)
    len2=N*llim+N; % Length of the auxiliary array v2
    v2=cumprod([gamma((-1-a)/2),(((-1-a)/2+(0:len2-2)))] ...
        ./[gamma((3+a)/2),((3+a)/2+(0:len2-2))]); % Store v_2(p)
    for l2=-floor(N/2):ceil(N/2)-1
        aux=(minus1l1.*v1(abs(Nl1+l2)+1))*(v2(abs(K-NL1-l2)+1) ...
            .*((1-a)*K.^2-2*K.*(NL1+l2)));
        Ma(1:ceil(N/2),k+1)=Ma(1:ceil(N/2),k+1)+exp(2i*l2*sj)*aux;
    end
    Ma(1:ceil(N/2),k+1)=(ca/(2*tan(a*pi/2)))...
        *sin(sj(1:ceil(N/2))).^(a-1).*Ma(1:ceil(N/2),k+1);
end
Ma(1:ceil(N/2),N-k+1)=conj(Ma(1:ceil(N/2),k+1)); % Apply symmetries.
Ma(ceil(N/2)+1:N,:)=conj(Ma(N-(ceil(N/2)+1:N)+1,:)); % Apply symmetries.
\end{lstlisting}

\subsection{Numerical experiments}

The main drawback of using $\Ma$ with $\alpha\not=1$ to compute \eqref{e:lklvslk} is that it must be applied to Fourier coefficients $\{\hat u_k\}$ such that $U(s) \equiv u(L\cot(s))$ can be regarded as $\pi$-periodic. On the other hand, the numerical methods developed in this paper may be applied to functions $u(x)$ such that $\lim_{x\to-\infty}u(x)\not= \lim_{x\to\infty}u(x)$, by subtracting a function $v(x)$, such that $w(x) = u(x) - v(x)$ satisfies $\lim_{x\to-\infty}w(x)= \lim_{x\to\infty}w(x)$ and the expression of the corresponding operator applied to $v(x)$ is explicitly known. For instance, if we know an analytic expression of $\Dag[v](x)$, then $\Dag[u](x) = \Dag[w](x) + \Dag[v](x)$, and we only need to approximate numerically $\Dag[w](x)$ by the techniques explained above. This is also valid for the numerical approximation of $\Da[u](x)$, $\barDa[u](x)$, $\partial_x\Db[u](x)$, $\partial_x\barDb[u](x)$, and for $(-\Delta)^{\alpha/2}u(x)$, too.

The idea of adding and subtracting a function $V(s)\equiv v(L\cot(s))$, to turn a nonperiodic function $U(s)\equiv u(L\cot(s))$ into a periodic one $W(s)\equiv w(L\cot(s)) \equiv u(L\cot(s)) - v(L\cot(s))$ has not been previously used in the $\alpha\not=1$ case, but has been recently tested in the $\alpha=1$ case in \cite{CuestaDelaHozGirona2023}, as an alternative approach to approximate numerically $(-\Delta)^{1/2}u(L\cot(s))$. However, unlike in the operators described in this paper for $\alpha\not=1$, $(-\Delta)^{1/2}\varphi_k(x)$ was known in \cite{CuestaDelaHozGirona2023} also for $k$ odd (see \eqref{e:fraclapphika}-\eqref{e:fraclapphiksb} in this paper), which enabled, e.g., to construct an interpolating function $V(s)$ such that $W(s)$ has a higher global regularity, and offered a good alternative to extending periodically $U(s)$ to $s\in[0,2\pi]$. On the other hand, as explained also in \cite{CuestaDelaHozGirona2023}, there are situations where the extension of $U(s)$ to $s\in[0,2\pi]$, and its Fourier expansion in terms of $e^{iks}$, using both even and odd values of $k$, is desirable, e.g., in evolution problems in which the asymptotic behavior of $u(x)$ is not exactly known as $x\to\pm\infty$.

In this paper, in order to illustrate this idea, we have taken $v(x)=\arctan(x)$ as an auxiliary function, because we can obtain explicit expressions for its Weyl-Marchaud derivatives by means of, e.g., Mathematica. More precisely, in order to calculate $\Da[\arctan](x)$ and $\barDa[\arctan](x)$ for $\alpha\in(0,1)$, we use respectively \eqref{Int:Dalpha}, with the change of variable $y = x - z$ inside the integral, and \eqref{Int:Dalpha:bar}, with the change of variable $y = x + z$. Therefore, we type
\begin{verbatim}
Integrate[ArcTan'[y]/(x-y)^a,{y,-Infinity,x}]/Gamma[1-a]
Integrate[ArcTan'[y]/(y-x)^a,{y,x,Infinity}]/Gamma[1-a]
\end{verbatim}
which yield, after some rewriting,
\begin{align*}
\Da[\arctan](x) & = \Gamma(\alpha)(1+x^2)^{-\alpha/2}\sin\left(\alpha\frac\pi2 + \alpha\arctan(x)\right),
	\cr
\barDa[\arctan](x) & = \Gamma(\alpha)(1+x^2)^{-\alpha/2}\sin\left(\alpha\frac\pi2 - \alpha\arctan(x)\right);
\end{align*}
note that Mathematica does not give exactly the same formula when $x > 0$ or when $x < 0$, but they are equivalent. Then, if we differentiate $\Da[\arctan](x)$ and $\barDa[\arctan](x)$ with respect to $x$, and substitute $\alpha$ by $\alpha-1$, we get, when $\alpha\in(1,2)$,
\begin{align*}
\partial_x\Db[\arctan](x) & = \Gamma(\alpha)(1+x^2)^{-\alpha/2}\sin\left(\alpha\frac\pi2 + \alpha\arctan(x)\right),
	\cr
\partial_x\barDb[\arctan](x) & = -\Gamma(\alpha)(1+x^2)^{-\alpha/2}\sin\left(\alpha\frac\pi2 - \alpha\arctan(x)\right),
\end{align*}
i.e., the expressions for $\Da[\arctan](x)$ and $\partial_x\Db[\arctan](x)$ are identical, whereas the expressions for $\barDa[\arctan](x)$ and $\partial_x\barDb[\arctan](x)$ differ only in the sign. Introducing these expressions in \eqref{RF:01:integral:representation} and \eqref{RF:integral:representation}, and applying basic trigonometric identities:
\begin{align*}
\Dag[\arctan](x) & = \frac{\Gamma(-\alpha)\Gamma(1+\alpha)\Gamma(\alpha)}{\pi}(1+x^2)^{-\alpha/2}
	\cr
& \qquad \times\Big[\sin\left( (\alpha-\gamma) \frac\pi2\right)\sin\left(\alpha\frac\pi2 + \alpha\arctan(x)\right)
	\cr
& \qquad \qquad - \sin\left( (\alpha+\gamma) \frac\pi2\right) \sin\left(\alpha\frac\pi2 - \alpha\arctan(x)\right)\Big]
	\cr
& = \Gamma(\alpha)(1+x^2)^{-\alpha/2}\sin\left(\gamma\frac\pi2 - \alpha\arctan(x)\right), \quad\alpha\in(0,2).
\end{align*}
Additionally, from \eqref{e:equivDzeroafraclap},
$$
(-\Delta)^{\alpha/2}\arctan(x) = -D_0^\alpha[\arctan](x) = \Gamma(\alpha)(1+x^2)^{-\alpha/2}\sin(\alpha\arctan(x)), \quad\alpha\in(0,2).
$$
Now, we are going to subtract $v(x) = \arctan(x)$ to $u(x) = \erf(x)$, in order to approximate numerically $\Da[\erf](x)$, $\barDa[\erf](x)$, $\partial_x\Db[\erf](x)$, $\partial_x\barDb[\erf](x)$, $\Dag[\erf](x)$ and $(-\Delta)^{\alpha/2}\erf(x)$; recall that
$$
\erf(x) = \frac{2}{\sqrt\pi}\int_{-\infty}^{x}e^{-y^2}dy.
$$
We observe that $\lim_{x\to-\infty}\arctan(x) = -\pi/2$, $\lim_{x\to\infty}\arctan(x) = \pi/2$, $\lim_{x\to-\infty}\erf(x) = -1$, whereas $\lim_{x\to\infty}\erf(x) = 1$, so we define accordingly $u(x) = \erf(x)$, $v(x) = (2/\pi)\arctan(x)$ and $w(x) = u(x) - v(x)$, getting $\lim_{x\to-\infty}w(x) = \lim_{x\to\infty}w(x) = 0$. Therefore, $W(s)\equiv w(L\cot(s))$ can be regarded as $\pi$-periodic.

In order to compute $\Da[\erf](x)$ and $\barDa[\erf](x)$ for $\alpha\in(0,1)$, we type now in Mathematica
\begin{verbatim}
	Integrate[Erf'[y]/(x-y)^a,{y,-Infinity,x}]/Gamma[1-a]
	Integrate[Erf'[y]/(y-x)^a,{y,x,Infinity}]/Gamma[1-a]
\end{verbatim}
which yield, after some rewriting,
\begin{align*}
\Da[\erf](x) & = \frac{2^\alpha}\pi\sin\left(\alpha\frac\pi2\right)\Gamma\left(\frac\alpha2\right)\,{}_1F_1\left(\frac\alpha2,\frac12,-x^2\right)
	\cr
& \qquad + \frac{2^{1+\alpha}}{\pi}\cos\left(\alpha\frac\pi2\right)\Gamma\left(\frac{1+\alpha}2\right)x\,{}_1F_1\left(\frac{1+\alpha}2,\frac32,-x^2\right),
	\cr
\barDa[\erf](x) & = \frac{2^\alpha}\pi\sin\left(\alpha\frac\pi2\right)\Gamma\left(\frac\alpha2\right)\,{}_1F_1\left(\frac\alpha2,\frac12,-x^2\right)
	\cr
& \qquad - \frac{2^{1+\alpha}}{\pi}\cos\left(\alpha\frac\pi2\right)\Gamma\left(\frac{1+\alpha}2\right)x\,{}_1F_1\left(\frac{1+\alpha}2,\frac32,-x^2\right),
\end{align*}
where ${}_1F_1$ is Kummer's confluent hypergeometric function:
\begin{equation*}
{}_{1}F_1(a;b;z)=\sum_{n=0}^{\infty}\frac{(a)_n}{(b)_n}\frac{z^n}{n!}.
\end{equation*}
Then, if we differentiate $\Da[\erf](x)$ and $\barDa[\erf](x)$ with respect to $x$, and substitute $\alpha$ by $\alpha-1$, we get, when $\alpha\in(1,2)$,
\begin{align*}
	\partial_x\Db[\erf](x) & = \frac{2^\alpha}\pi\sin\left(\alpha\frac\pi2\right)\Gamma\left(\frac\alpha2\right)\,{}_1F_1\left(\frac\alpha2,\frac12,-x^2\right)
	\cr
	& \qquad + \frac{2^{1+\alpha}}{\pi}\cos\left(\alpha\frac\pi2\right)\Gamma\left(\frac{1+\alpha}2\right)x\,{}_1F_1\left(\frac{1+\alpha}2,\frac32,-x^2\right),
	\cr
	\partial_x\barDb[\erf](x) & = -\frac{2^\alpha}\pi\sin\left(\alpha\frac\pi2\right)\Gamma\left(\frac\alpha2\right)\,{}_1F_1\left(\frac\alpha2,\frac12,-x^2\right)
	\cr
	& \qquad + \frac{2^{1+\alpha}}{\pi}\cos\left(\alpha\frac\pi2\right)\Gamma\left(\frac{1+\alpha}2\right)x\,{}_1F_1\left(\frac{1+\alpha}2,\frac32,-x^2\right),
\end{align*}
i.e., the expressions for $\Da[\erf](x)$ and $\partial_x\Db[\erf](x)$ are identical, whereas the expressions for $\barDa[\erf](x)$ and $\partial_x\barDb[\erf](x)$ differ only in the sign. Introducing these expressions in \eqref{RF:01:integral:representation} and \eqref{RF:integral:representation}, we conclude after some simplification that
\begin{align*}
\Dag[\erf](x) & = \frac{2^{\alpha}}\pi\Gamma\left(\frac\alpha2\right)\sin\left(\gamma\frac\pi2\right){}_1F_1\left(\frac\alpha2,\frac12,-x^2\right)
	\cr
& \qquad - \frac{2^{1+\alpha}}\pi\Gamma\left(\frac{1+\alpha}2\right)\cos\left(\gamma\frac\pi2\right)x\,{}_1F_1\left(\frac{1+\alpha}2,\frac32,-x^2\right), \quad\alpha\in(0,2).
\end{align*}
Additionally, from \eqref{e:equivDzeroafraclap},
$$
(-\Delta)^{\alpha/2}\erf(x) = -D_0^\alpha[\erf](x) = \frac{2^{1+\alpha}}\pi\Gamma\left(\frac{1+\alpha}2\right)x\,{}_1F_1\left(\frac{1+\alpha}2,\frac32,-x^2\right), \quad\alpha\in(0,2).
$$
In Listing \ref{code:testfractionalLaplacian}, we have approximated numerically $\Da[\erf](x)$, $\barDa[\erf](x)$, $D_\gamma^\alpha[\erf](x)$ and $(-\Delta)^{\alpha/2}\erf(x)$, taking $\alpha = 0.62$ and $\gamma = 0.49$, $N = 256$, $L = 1.1$ and $l_{lim} = 100$, and compared them with their exact expressions; the errors in $L^\infty$-norm are, respectively, $1.6029\times10^{-14}$, $1.6036\times10^{-14}$, $2.0241\times10^{-14}$ and $2.7544\times10^{-14}$.

\begin{lstlisting}[label=code:testfractionalLaplacian, language=Matlab, basicstyle=\footnotesize, caption = {Numerical approximation of the operators applied to $\erf(x)$}]
clear
tic
a=0.62; % Parameter alpha
g=0.49; % Parameter gamma
N=256; % Number of nodes N
L=1.1; % Scaling factor L
llim=100; % l_{lim}
sj=pi*(1:2:2*N-1)'/(2*N); % Nodes s_j in (0,pi)
xj=L*cot(sj); % Nodes x_j in R
u=erf(xj); % Define the function u(x) to which we apply the operators.
v=(2/pi)*atan(xj); % Define the auxiliary function v(x).
w=u-v; % Define w(x). W(Lcot(s)) can be regarded as pi-periodic.
% Fourier coefficents \hat w_k:
hatwk=(exp(-1i*pi*[0:ceil(N/2)-1 -floor(N/2):-1]'/N)/N).*fft(w);
hatwk(abs(hatwk)<eps)=0; % Apply the Krasny filter.
% Create the matrix Ma to approximate (-Delta)^(alpha/2):
Ma=fraclapMa(N,a,llim)/L^a; % M_alpha
% Create the matrix to approximate D^alpha:
MDa=zeros(N,N);
MDa(:,2:ceil(N/2))=exp(-1i*a*pi/2)*Ma(:,2:ceil(N/2));
MDa(:,ceil(N/2)+1:N)=exp(1i*a*pi/2)*Ma(:,ceil(N/2)+1:N);
% Create the matrix to approximate \bar D^alpha:
MbarDa=zeros(N,N);
MbarDa(:,2:ceil(N/2))=-exp(1i*a*pi/2)*Ma(:,2:ceil(N/2));
MbarDa(:,ceil(N/2)+1:N)=-exp(-1i*a*pi/2)*Ma(:,ceil(N/2)+1:N);
% Create the matrix to approximate D_gamma^alpha:
MDag=zeros(N,N);
MDag(:,2:ceil(N/2))=-exp(1i*g*pi/2)*Ma(:,2:ceil(N/2));
MDag(:,ceil(N/2)+1:N)=-exp(-1i*g*pi/2)*Ma(:,ceil(N/2)+1:N);
% In this code, Da_ denotes D^alpha, barDa_ denotes bar D^alpha,
% Dag_ denotes D_gamma^alpha and fraclap_ denotes (-Delta)^(alpha/2).
% Exact expressions of the operators applied to v(x):
Da_v=(2/pi)*gamma(a)*(1+xj.^2).^(-a/2).*sin(a*pi/2+a*atan(xj));
barDa_v=(2/pi)*gamma(a)*(1+xj.^2).^(-a/2).*sin(a*pi/2-a*atan(xj));
Dag_v=(2/pi)*gamma(a)*(1+xj.^2).^(-a/2).*sin(g*pi/2-a*atan(xj));
fraclap_v=(2/pi)*gamma(a)*(1+xj.^2).^(-a/2).*sin(a*atan(xj)); 
% Numerical approximations of the operators applied to w(x):
Da_w=MDa*hatwk;
barDa_w=MbarDa*hatwk;
Dag_w=MDag*hatwk;
fraclap_w=Ma*hatwk;
% Numerical approximations of the operators applied to u(x):
Da_u=Da_w+Da_v;
barDa_u=barDa_w+barDa_v;
Dag_u=Dag_w+Dag_v;
fraclap_u=fraclap_w+fraclap_v;
if isreal(u) % If u(x) is real, so are the operators applied to it.
    Da_u=real(Da_u);
    barDa_u=real(barDa_u);
    Dag_u=real(Dag_u);
    fraclap_u=real(fraclap_u);
end
% Exact expressions of the operators applied to erf(x):
Da_erf=2^a/pi*sin(a*pi/2)*gamma(a/2)*hypergeom(a/2,1/2,-xj.^2) ...
    +2^(1+a)/pi*cos(a*pi/2)*gamma((1+a)/2)*xj.*hypergeom((1+a)/2,3/2,-xj.^2);
barDa_erf=2^a/pi*sin(a*pi/2)*gamma(a/2)*hypergeom(a/2,1/2,-xj.^2) ...
    -2^(1+a)/pi*cos(a*pi/2)*gamma((1+a)/2)*xj.*hypergeom((1+a)/2,3/2,-xj.^2);
Dag_erf=2^a/pi*gamma(a/2)*sin(g*pi/2)*hypergeom(a/2,1/2,-xj.^2) ...
    -2^(1+a)/pi*gamma((1+a)/2)*cos(g*pi/2)*xj.*hypergeom((1+a)/2,3/2,-xj.^2);
fraclap_erf=2^(1+a)/pi*gamma((1+a)/2)*xj.*hypergeom((1+a)/2,3/2,-xj.^ 2);
% Errors of the numerical approximations in L^infty-norm:
norm(Da_u-Da_erf, "inf")
norm(barDa_u-barDa_erf, "inf")
norm(Dag_u-Dag_erf, "inf")
norm(fraclap_u-fraclap_erf, "inf")
toc % Elapsed time
\end{lstlisting}

At this point, it is interesting to see the effect of $L$, whose choice is always delicate. Indeed, although there are some theoretical results \cite{Boyd1982}, a good election of $L$ depends on many factors: number of points, class of functions, type of problem, etc., and a good working rule of thumb seems to be that the absolute value of a given function at the extreme grid points is lower than a certain threshold. Moreover, according to the numerical experiments in \cite{CuestaDelaHozGirona2023}, given a function $u(x)$, the global regularity of its corresponding $U(s)\equiv u(L\cot(s))$ in all $\mathbb R$ (or of $W(s) = U(s) - V(s)$, when it is applicable) plays a fundamental role. For instance, if we have only that $U(s)\in\mathbb C^0(\mathbb R)$, in general larger values of $L$ and of $N$ are required to attain the maximum possible accuracy, and the optimum value of $L$, which grows with $N$, is more difficult to capture, whereas more globally regular functions $U(s)$ require lower values of $N$ to attain the maximum possible accuracy, and this happens for a larger range of values of $L$. In our case, we have approximated numerically $D_\gamma^\alpha[\erf](x)$ for $\alpha = 1.37$, $\gamma = 0.58$, $l_{lim} = 100$, taking $N\in\{2^3, 2^4, \ldots, 2^{14}\}$ and $L\in\{0.01, 0.02, \ldots, 10\}$, and plotted in semilogarithmic scale the $L^\infty$-norm of the errors on the left-hand side of Figure \ref{f:errN16384}. While the results with $L$ close to $0^+$ are nonsensical, accuracies of the order of $\mathcal O(10^{-14})$ are achieved already for $N = 128$, and for a number of values of $L$, which is in agreement with the fact that $W(s) = \erf(L\cot(s)) - (2/\pi)\arctan(L\cot(s))\in\mathbb C^\infty(\mathbb R)$. Observe also that for $N = 512$ and larger, the results are virtually indistinguishable.

\begin{figure}
	\centering
	\includegraphics[width=0.5\textwidth, clip=true]{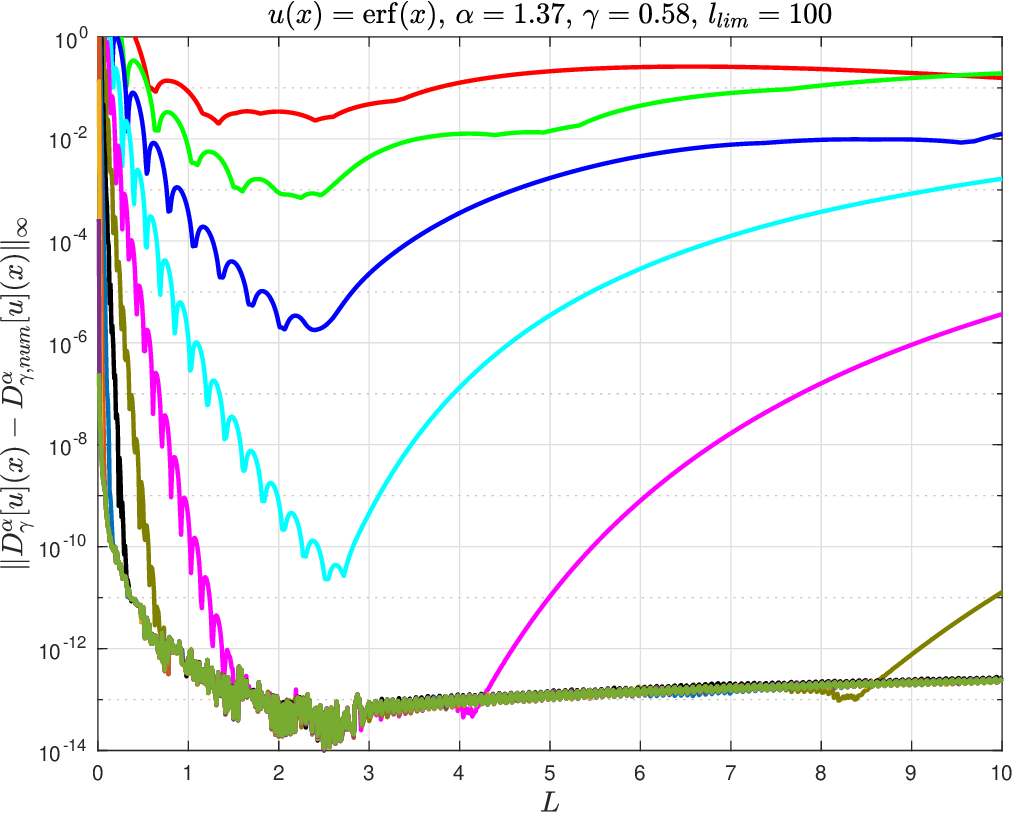}\includegraphics[width=0.5\textwidth, clip=true]{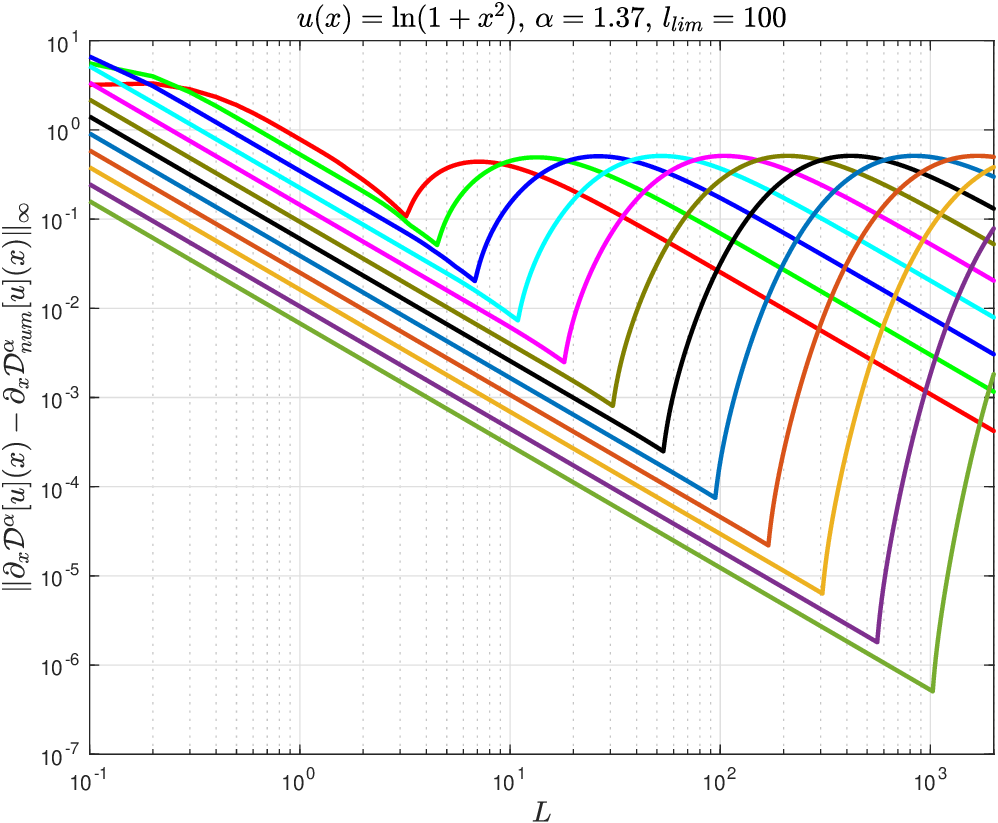}
	\includegraphics[width=\textwidth, clip=true]{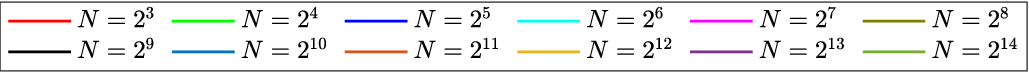}
	\caption{Left: Errors in the numerical approximation of $D_\gamma^\alpha[\erf](x)$ versus $L\in\{0.01, 0.02, \ldots, 10\}$, for $\alpha = 1.37$, $\gamma = 0.58$, $l_{lim} = 100$ and $N\in\{2^3, 2^4, \ldots, 2^{14}\}$. Right: Errors in the numerical approximation of $\partial_x\Db[\ln(1+x^2)]$ versus $L\in\{0.1, 0.2, \ldots, 1000\}$, for $\alpha = 1.37$,  $l_{lim} = 100$ and $N\in\{2^3, 2^4, \ldots, 2^{14}\}.$ }
	\label{f:errN16384}
\end{figure}

On the other hand, we have also considered $u(x) = \ln(1 + x^2)$, i.e., a function that is not bounded and has logarithmic growth as $x\to\pm\infty$; note that $\partial_x\Da[\ln(1+x^2)]$ was also approximated numerically in \cite{delahozcuesta2016} for $\alpha\in(0,1)$, which is equivalent to $\partial_x\Db[\ln(1+x^2)]$ for $\alpha\in(1,2)$ in this paper. In order to compute $\Da[\ln(1+x^2)]$ and $\barDa[\ln(1+x^2)](x)$ for $\alpha\in(0,1)$, we can use again Mathematica, and type
\begin{verbatim}
Integrate[D[Log[1 + y^2], y]/(x-y)^a,{y,-Infinity,x}]/Gamma[1-a]
Integrate[D[Log[1 + y^2], y]/(y-x)^a,{y,x,Infinity}]/Gamma[1-a]
\end{verbatim}
which yield, after some rewriting,
\begin{align*}
	\Da[\ln(1+x^2)] & = -2\Gamma(\alpha)(1+x^2)^{-\alpha/2}\cos\left(\alpha\frac\pi2 + \alpha\arctan(x)\right),
	\cr
	\barDa[\ln(1+x^2)] & = 2\Gamma(\alpha)(1+x^2)^{-\alpha/2}\cos\left(\alpha\frac\pi2 - \alpha\arctan(x)\right),
\end{align*}
Then, if we differentiate $\Da[\ln(1+x^2)]$ and $\barDa[\ln(1+x^2)]$ with respect to $x$, and substitute $\alpha$ by $\alpha-1$, we get, when $\alpha\in(1,2)$,
\begin{align*}
	\partial_x\Db[\ln(1+x^2)] & = -2\Gamma(\alpha)(1+x^2)^{-\alpha/2}\cos\left(\alpha\frac\pi2 + \alpha\arctan(x)\right),
	\cr
	\partial_x\barDb[\ln(1+x^2)] & = -2\Gamma(\alpha)(1+x^2)^{-\alpha/2}\cos\left(\alpha\frac\pi2 - \alpha\arctan(x)\right),
\end{align*}
i.e., the expressions for $\Da[\ln(1+x^2)]$ and $\partial_x\Db[\ln(1+x^2)]$ are identical, whereas the expressions for $\barDa[\ln(1+x^2)]$ and $\partial_x\barDb[\ln(1+x^2)]$ differ only in the sign. Introducing these expressions in \eqref{RF:01:integral:representation} and \eqref{RF:integral:representation}, we conclude after some simplification that
\begin{align*}
\Dag[\ln(1+x^2)] & = -2\frac{\Gamma(-\alpha)\Gamma(1+\alpha)\Gamma(\alpha)}{\pi}(1+x^2)^{-\alpha/2}\bigg[\sin\left((\alpha-\gamma) \frac\pi2\right)\cos\left(\alpha\frac\pi2 + \alpha\arctan(x)\right)
	\cr
& \qquad +  \sin\left( (\alpha+\gamma) \frac\pi2\right)\cos\left(\alpha\frac\pi2 - \alpha\arctan(x)\right)\bigg],
	\cr
& = 2\Gamma(\alpha)(1+x^2)^{-\alpha/2}\cos\left(\gamma\frac\pi2 - \alpha\arctan(x)\right), \quad\alpha\in(0,2).
\end{align*}
Additionally, from \eqref{e:equivDzeroafraclap},
$$
(-\Delta)^{\alpha/2}\ln(1+x^2) = -D_0^\alpha[\ln(1+x^2)] = -2\Gamma(\alpha)(1+x^2)^{-\alpha/2}\cos(\alpha\arctan(x)), \quad\alpha\in(0,2).
$$
Applying the operators to $u(x) = \ln(1+x^2)$ is very challenging, because $U(s) = \ln(1+L^2\cot^2(s))$ is not even continouous on $\mathbb R$, so we can expect rather poor results. For instance, taking $\alpha = 1.12$, $\gamma= 0.83$, $N = 256$, $L = 30$ and $l_{lim}  = 100$, the errors in $L^\infty$ for the numerical approximations of $\partial_x\Db[\ln(1+x^2)]$, $\partial_x\barDb[\ln(1+x^2)]$, $D_\gamma^\alpha[\ln(1+x^2)]$ and $(-\Delta)^{\alpha/2}\ln(1+x^2)$ are, respectively, $5.9765\times10^{-4}$, $5.9765\times10^{-4}$, $5.9097\times10^{-4}$ and $1.1931\times10^{-4}$, and the results are similar for $\alpha\in(1,2)$, and $|\gamma| \leq \min\{\alpha, 2-\alpha\}$. However, we remark that these results can be much improved by increasing $N$ and choosing carefully $L$. To illustrate this, we have plotted on the right-hand side of Figure \ref{f:errN16384}, in log-log scale, $\Da[\ln(1+x^2)]$ for $\alpha = 1.37$, $l_{lim} = 100$, taking $N\in\{2^3, 2^4, \ldots, 2^{14}\}$ and $L\in\{0.1, 0.2, \ldots, 1000\}$. The experiments seem to suggest that the optimum value of $L$ grows exponentially with $N$, and the lowest error, $5.0735\times10^{-7}$, is achieved for $N = 16384$ and $L = 1025.4$. Furthermore, when $N = 16384$ and $L\in[625.7, 1051.3]$, the errors are of the order of $\mathcal O(10^{-7})$. These results are in our opinion very remarkable, and show the adequacy of the methods in this paper to deal with unbounded, discontinuous functions $U(s)$.

\subsection{An evolution example}

To finish this paper, we have simulated the nonlinear Riesz-Feller fractional diffusion equation \eqref{e:fisher0} for $f(u) = u(1-u)$. More precisely, we have considered the following initial value problem:
\begin{equation}
	\label{e:fisher}
	\left\{
	\begin{aligned}
		& u_t = D_\gamma^\alpha u + u(1-u), \quad x\in\mathbb R,\ t \ge 0,
		\cr
		& u(x, 0) = \left(\frac{1}{2} - \frac{x}{2\sqrt{1 + x^2}}\right)^{\alpha/2}.
	\end{aligned}
	\right.
\end{equation}
Then, according to \cite{CabreRoquejoffre2013}, this equation gives rise to front propagation solutions that travel towards $u = 0$ with a wave speed $c(t)$ exponentially increasing in time, i.e., $c(t)\sim e^{\sigma t}$; more precisely, in the case of slow decaying initial conditions, $\sigma = f'(0) / \alpha$ or faster, whereas, in the case of fast decaying initial conditions, $\sigma = f'(0) / (1 + \alpha)$. Therefore, a very large spatial domain or the whole domain is required, in order to simulate adequately those fronts. Note that this behavior was reproduced numerically in \cite{cayamacuestadelahoz2021}, for the case $\gamma = 1$ (i.e., $(-\Delta)^{\alpha/2}$ was considered instead of $D_\gamma^\alpha$), and in \cite{CuestaDelaHozGirona2023}, for $\gamma = 0$ and $\alpha = 1$ (i.e., the half Laplacian $(-\Delta)^{1/2}$ was considered).

In this paper, we have considered an initial condition $u(x, 0)$ in \eqref{e:fisher} that decays slowly according to \cite{CabreRoquejoffre2013}. Moreover, we have chosen $\alpha = 1.37$, $N = 16384$, $l_{lim} = 100$, and different values of gamma, namely $\gamma\in\{-0.63, -0.625, \ldots, 0.625, 0.63\}\subset[\alpha-2,2-\alpha]$, in order to see what the effect of $\gamma$ on the solutions is. In order to advance in time, we have implemented a classical fourth-order Runge-Kutta method, taking $\Delta t = 0.05$, and $t\in[0,22]$.

\begin{figure}[!htbp]
	\centering
	\includegraphics[width=0.5\textwidth, clip=true]{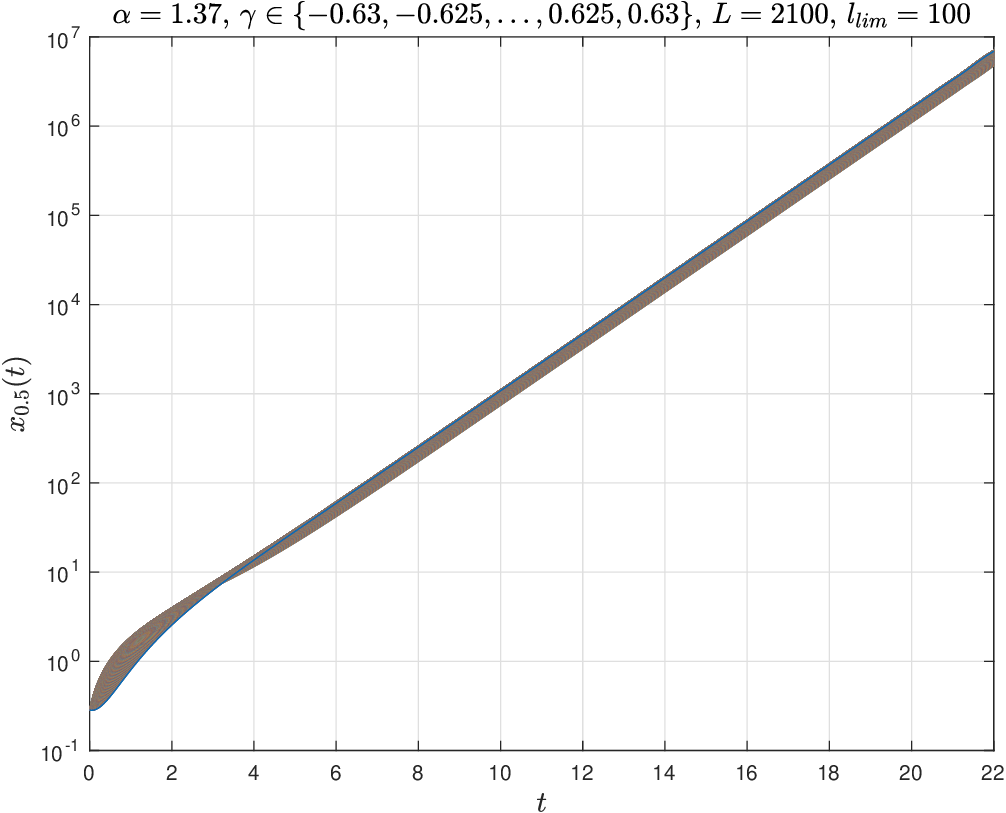}\includegraphics[width=0.5\textwidth, clip=true]{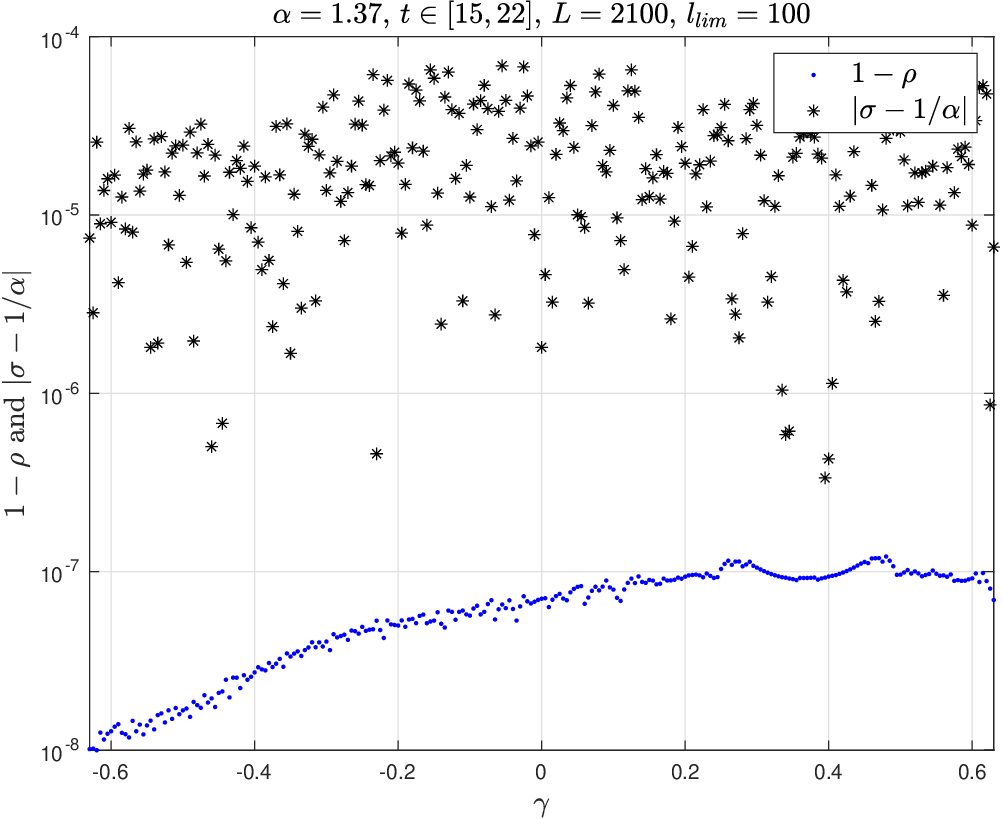}
	\caption{Left: $x_0(t)$ versus $t$. Right: $1 - \rho$ (blue dots) versus $t$, where $\rho$ is the Pearson correlation coefficient, and $|\sigma - 1/\alpha|$ (black stars) versus $t$, where $\sigma$ is the slope of the corresponding straight line on the left-hand side, such that the speed $c(t)\sim e^{\sigma t}$. In both sides, $\alpha = 1.37$, $\gamma \in \{-0.63, -0.625, \ldots, 0.625, 0.63\}$, $L = 2100$, $l_{lim} = 100$. }
	\label{f:slopes}
\end{figure}

We have found that, for all the values of $\gamma$ considered, $\lim_{x\to-\infty}u(x,t) = 1$ and $\lim_{x\to-\infty}u(x,t) = 0$, so, using $v(x) = 1/2 - \arctan(x)/\pi$, we can apply the techniques in this paper to $w(x, t) = u(x, t) - v(x)$, as it satisfies $\lim_{x\to\pm\infty}w(x, t) = 0$. Furthermore, for those values of $\gamma$, we have again front propagating solutions with exponentially growing speed. In order to capture the speed of a given front, we have 
tracked the evolution of the value of $x$ such that $u(x, t) = 0.5$, for a given $t$; this value, denoted as $x_{0.5}(t)$, can be computed by means of a bisection method, together with spectral interpolation (see \cite{cayamacuestadelahoz2021,CuestaDelaHozGirona2023}). Since $x_{0.5}(t)$ becomes quickly very large, it is necessary to take large values of $L$; in this regard, we have found that $L = 2100$ is enough for our purposes in all the cases.

On the left-hand side of Figure \ref{f:slopes}, we have plotted $x_{0.5}(t)$ versus $t\in[0,22]$ in semilogarithmic scale. Except for the initial times, when the exponential velocity regime has not still been achieved, the results reveal what seems to be parallel straight lines, that, when plotted together, resemble a narrow rectangle. To confirm this, we have computed the slopes $\sigma$ of those lines by the method of least squares, taking $t\in[15, 21]$, together with the Pearson correlation coefficient $\rho$. In all cases, $\rho$ is equal to one, except for a tiny error, and all the values of $\sigma$ are very close to $1 / \alpha$; more precisely, $1 - \rho\in[9.9629\times10^{-9}, 1.2163\times10^{-7}]$, and $|\sigma - 1 / \alpha|\in[3.3561\times10^{-7}, 6.8708\times10^{-5}]$. This can be seen on the right-hand side of Figure \ref{f:slopes}, where we have plotted $1 - \rho$ (blue dots) versus $t$ and $|\sigma - 1/\alpha|$ (black stars) versus $t$. In view of the results, there is in our opinion solid evidence that the speed $c(t)$ grows exponentially in time, and that $c(t)\sim e^{\sigma t} = e^{t / \alpha}$, for all $\gamma$, which is in agreement with the theory in \cite{CabreRoquejoffre2013}.

Finally, in Figure \ref{f:evolu}, we have plotted the evolution of \eqref{e:fisher} for $\gamma = -0.63$, and $t\in\{0, 0.5, \ldots, 21.5, 22\}$; the left-hand side shows the appearance of the front, that travels from the stable state $u = 1$ into the unstable one $u = 0$, whereas the right-hand side, in semilogarithmic scale, illustrates that the speed of the front becomes exponential as $t$ grows; indeed, except for the initial values of $t$, the curves of $u(x, t)$ are parallel with high accuracy, and we are able to capture this behavior for values of $x$ of the order of $\mathcal O(10^6)$.

\begin{figure}[!htbp]
	\centering
	\includegraphics[width=0.5\textwidth, clip=true]{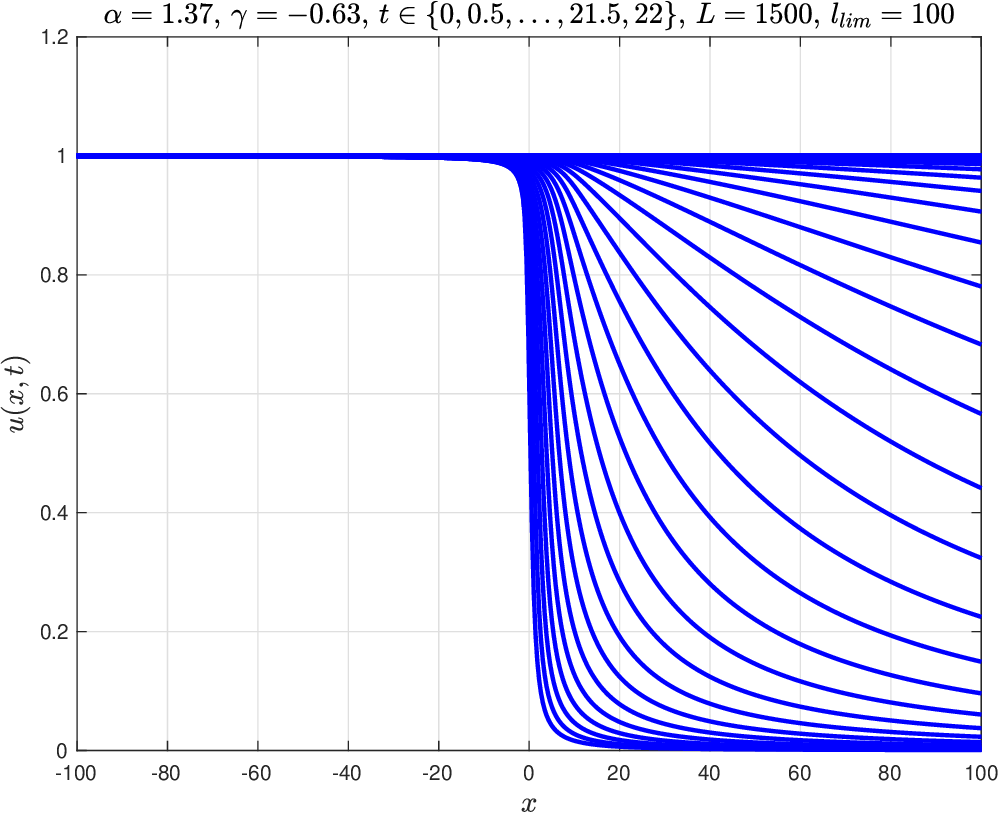}\includegraphics[width=0.5\textwidth, clip=true]{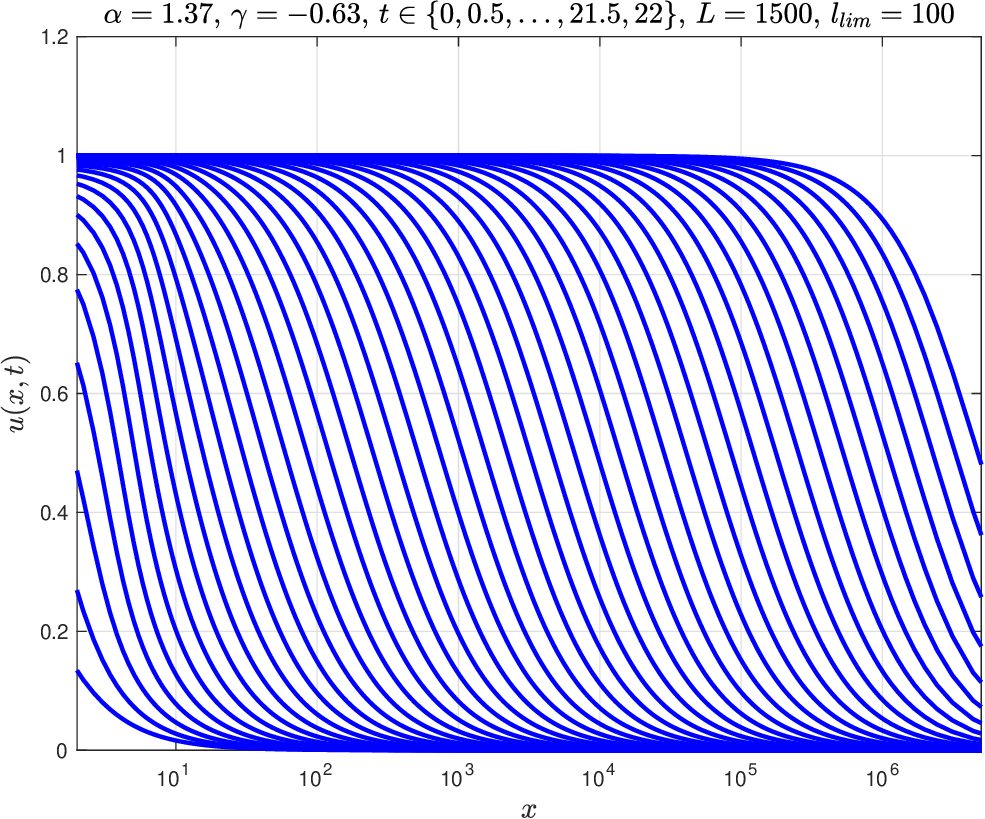}
	\caption{Evolution of \eqref{e:fisher} for $\alpha = 1.37$, $\gamma = -0.63$, $t\in\{0, 0.5, \ldots, 21.5, 22\}$, $L = 2100$, $l_{lim} = 100$. Left: $x\in[-100, 100]$. Right: $x\in[1, 10^6]$, and $u(x, t)$ appears in semilogarithmic scale.}
	\label{f:evolu}
\end{figure}

Doing a careful study of \eqref{e:fisher} lies beyond the scope of this paper. However, this example illustrates the adequacy of our numerical method to simulate over very large domains evolution equations involving Riesz-Feller operators.

\section*{Acknowledgments}

This work was partially supported by the research group grant IT1615-22 funded by the Basque Government, and by the project PID2021-126813NB-I00 funded by MCIN/AEI/10.13039/501100011033 and by ``ERDF A way of making Europe''. Ivan Girona was also partially supported by the MICINN PhD grant PRE2019-090478.

\end{document}